\documentclass[11pt,titlepage]{amsart}
\usepackage{amssymb}
\usepackage{amsfonts}
\usepackage[colorlinks=true, pdfstartview=FitV, linkcolor=blue,
citecolor=blue, urlcolor=blue]{hyperref}
\usepackage{amsmath}
\usepackage{setspace}

\calclayout
\numberwithin{equation}{section}
\allowdisplaybreaks
\newtheorem{theorem}{Theorem}[section]
\newtheorem{corollary}[theorem]{Corollary}

\newtheorem{lemma}[theorem]{Lemma}
\newtheorem{proposition}[theorem]{Proposition}

\newtheorem{example}[theorem]{Example}

\newtheorem{condition}[theorem]{Condition}

\theoremstyle{definition}
\newtheorem{definition}[theorem]{Definition}
\theoremstyle{remark}
\newtheorem{remark}[theorem]{Remark}

\def\func#1{\mathop{\rm #1}}%
\def\limfunc#1{\mathop{\rm #1}}%

\begin{document}
\title[Fractional Powers of Non-symmetric Operators]{Harnack's Inequality
and A Priori Estimates for Fractional Powers of Non-symmetric Differential
Operators}
\author{H. Aimar}
\address{Instituto de Matem\'{a}tica Aplicada del Litoral\\
CCT, CONICET, Santa Fe, Argentina}
\email{haimar@santafe-conicet.gov.ar}
\author{G. Beltritti}
\address{Instituto de Matem\'{a}tica Aplicada del Litoral\\
CCT, CONICET, Santa Fe, Argentina}
\email{gbeltritti@santafe-conicet.gov.ar}
\author{I. G\'{o}mez}
\address{Instituto de Matem\'{a}tica Aplicada del Litoral\\
CCT, CONICET, Santa Fe, Argentina}
\email{vanagomez@santafe-conicet.gov.ar}
\author{C. Rios}
\address{University of Calgary\\
Calgary, AB T2N1N4, Canada}
\email{crios@math.ucalgary.ca}
\date{%
\today%
}

\begin{abstract}
We obtain a new general extension theorem in Banach spaces for operators
which are not required to be symmetric, and apply it to obtain Harnack
estimates and a priori regularity for solutions of fractional powers of
several second order differential operators. These include weighted elliptic
and subellitptic operators in divergence form (nonnecessarily self-adjoint),
and nondivergence form operators with rough coefficients. We utilize the
reflection extension technique introduced by Caffarelli and Silvestre.
\end{abstract}

\maketitle

\section{Introduction}

We consider several nonnegative second order differential operators $L$
densely defined on a Banach space. Under different structural assumptions we
will establish Harnack estimates for nonnegative solutions of fractional
powers of $L$ as consequence of existing Harnack estimates for an associated
extended problem. This technique was pioneered by Caffarelli and Silvestre 
\cite{caffarelli-silvestre-07} for the fractional Laplacian, and it has
already been multiplied into a great number of applications. In a nutshell,
this is how the technique works to obtain results for the square root of for
the Laplacian operator: if $u\left( x,y\right) $ is the smooth bounded
solution of the extension problem%
\begin{eqnarray*}
u\left( x,0\right)  &=&f\left( x\right) \qquad \text{for }x\in \mathbb{R}%
^{n}, \\
\Delta u\left( x,y\right)  &=&0\qquad \text{for }x\in \mathbb{R}^{n}\text{%
\quad and}\quad y>0,
\end{eqnarray*}%
then that $\left( -\Delta _{x}\right) ^{1/2}f\left( x\right) =-u_{y}\left(
x,0\right) $ (the Dirichlet to Neumann map) as it has long been known. On
the other hand, if $\left( -\Delta _{x}\right) ^{1/2}f\equiv 0$ in an open
set $\Omega \subset \mathbb{R}^{n}$, then the extended function 
\begin{equation*}
u\left( x,y\right) =\left\{ 
\begin{array}{cc}
u\left( x,y\right)  & \qquad y\geq 0 \\ 
u\left( x,-y\right)  & \qquad y<0%
\end{array}%
\right. 
\end{equation*}%
is a solution of $\Delta u=0$ in $\Omega \times \mathbb{R}$. If moreover $%
f\geq 0$ in $\Omega $ it follows (by Poisson's formula) that $u\geq 0$ in $%
\Omega \times \mathbb{R}$ and therefore $u$ satisfies Harnack's inequality
there. As a consequence, $f\left( x\right) =u\left( x,0\right) $ satisfies
Harnack's inequality in $\Omega $, thus obtaining Harnack's estimates for
nonnegative solutions of $\left( -\Delta _{x}\right) ^{1/2}f=0$. This
principle was extended in \cite{caffarelli-silvestre-07} to other powers $%
\sigma $ of the Laplacian by considering the extension problem%
\begin{eqnarray*}
u\left( x,0\right)  &=&f\left( x\right) \qquad \qquad \text{on }\mathbb{R}%
^{n} \\
-\Delta _{x}u+\frac{1-2\sigma }{y}u_{y}+u_{yy} &=&0\qquad \qquad \text{in }%
\mathbb{R}^{n}\times \left( 0,\infty \right) ,
\end{eqnarray*}%
and proving that for constants $c_{\sigma ,1},c_{\sigma ,2}$ 
\begin{equation*}
\lim_{y\rightarrow 0^{+}}\frac{u\left( x,y\right) -u\left( x,0\right) }{%
y^{2\sigma }}=c_{\sigma ,1}~\left( -\Delta \right) ^{\sigma }f\left(
x\right) =c_{\sigma ,2}\lim_{y\rightarrow 0^{+}}y^{1-2\sigma }u_{y}\left(
x,y\right) .
\end{equation*}%
The operator $-\Delta _{x}+\frac{1-2\sigma }{y}\frac{\partial }{\partial y}+%
\frac{\partial ^{2}}{\partial y^{2}}=-\func{div}_{\left( x,y\right) }\left(
y^{1-2\sigma }\nabla _{\left( x,y\right) }\cdot \right) $ satisfies a
Harnack inequality as consequence of the Fabes-Kenig-Serapioni results for
weighted elliptic operators \cite{fabes-kenig-serapioni82}, and the result
follows for powers $0<\sigma <1$ in the same way as for $\sigma =\frac{1}{2}$%
.

For a general operator $L$ defined in $\mathbb{R}^{n}$ this procedure may be
divided in three main steps:

\begin{enumerate}
\item \label{step-1}Solving the extension problem in $\mathbb{R}_{+}^{n+1}$.
Find a solution of%
\begin{equation*}
\begin{array}{cl}
\mathcal{L}u=\left( -L_{x}+\frac{1-2\sigma }{y}\partial _{y}+\partial
_{y}^{2}\right) u=0 & \qquad \left( x,y\right) \in \mathbb{R}^{n}\times
\left( 0,\infty \right) , \\ 
u\left( x,0\right) =f\left( x\right) & \qquad x\in \mathbb{R}^{n}.%
\end{array}%
\end{equation*}%
Prove that%
\begin{equation*}
\lim_{y\rightarrow 0^{+}}\frac{u\left( x,y\right) -u\left( x,0\right) }{%
y^{2\sigma }}=c_{\sigma ,1}~L^{\sigma }f\left( x\right) =c_{\sigma
,2}\lim_{y\rightarrow 0^{+}}y^{1-2\sigma }u_{y}\left( x,y\right) .
\end{equation*}

\item \label{step-2}If $Lf\equiv 0$ in an open set $\Omega \subset \mathbb{R}%
^{n}$, show that the extended function%
\begin{equation*}
u\left( x,y\right) =\left\{ 
\begin{array}{cc}
u\left( x,y\right) & y\geq 0 \\ 
u\left( x,-y\right) & y<0%
\end{array}%
\right.
\end{equation*}%
is a solution of $\mathcal{L}u=0$ in $\Omega \times \mathbb{R}$.

\item \label{step-3}Establish (from existing literature or otherwise) that
nonnegative solutions to $\mathcal{L}u=0$ satisfy a Harnack's inequality, or
have a priori estimates which, by restriction, are also valid for solutions
to $L^{\sigma }f=0$.
\end{enumerate}

In \cite{torrea-stinga-10} Torrea and Stinga established (\ref{step-1}) for
a very general class of second order self-adjoint linear differential
operators, and applied the technique to obtain Harnack's estimates for
solutions of the fractional harmonic oscillator operator $H^{\sigma }=\left(
-\Delta +\left\vert x\right\vert ^{2}\right) ^{\sigma }$.

In this work we extend (\ref{step-1}) to operators that might not be
self-adjoint, and apply the results to a variety of important examples. In
particular, the application of the extension techniques and the existence of
the functional calculus to nondivergence form operators is new in this level
of generality. The paper is organized as follows. In the remainder of this
introduction we list three different applications of our main extension
theorem. In Section \ref{section-main} we state and prove our main result,
the extension theorem for closed operators on Banach spaces. In Section \ref%
{section-applications} we prove the three applications presented here. We
note that our result for subelliptic operators in Section \ref{section-noni}
includes a bigger class of operators than the diagonal ones presented in
Theorem \ref{theorem-noniso} below, and that Theorems \ref{theorem-weighted}
and \ref{theorem-nond} may also be extended to operators with drift and zero
order terms. Finally, in the Appendix, Section \ref{section-appendix}, we
include some basic facts about non-symmetric Dirichlet forms and functional
calculus for easier reference.

Our first application to illustrate the utility of our main extension
theorem is to weighted elliptic operators. Given $0<\lambda \leq \Lambda
<\infty ,$ let $\mathcal{F}_{n}\left( \lambda ,\Lambda \right) $ denote the
set of all $n\times n$ real valued matrix functions $\mathbf{A}\left(
x\right) $ such that%
\begin{equation}
\mathbf{A}\left( x\right) \xi \cdot \xi \geq \lambda \left\vert \xi
\right\vert ^{2}\quad \text{and}\quad \left\vert \mathbf{A}\left( x\right)
\xi \cdot \eta \right\vert \leq \Lambda \left\vert \xi \right\vert
\left\vert \eta \right\vert \text{ for all }x,\xi ,\eta \in \mathbb{R}^{n},
\label{ellipticity}
\end{equation}%
that is, $\mathcal{F}_{n}\left( \lambda ,\Lambda \right) $ is the set of
real valued $n\times n$ matrices which eigenvalues lie in the interval $%
\left[ \lambda ,\Lambda \right] $.

A weight $w$ in the Muckemphout class $A_{2}$ is a nonnegative locally
integrable function in $\mathbb{R}^{n}$ such that%
\begin{equation*}
\left[ w\right] _{A_{2}}=\sup_{x\in \mathbb{R}^{n},r>0}\left( \frac{1}{%
\left\vert B_{r}\left( x\right) \right\vert }\int\limits_{B_{r}\left(
x\right) }w\left( y\right) ~dy\right) ^{\frac{1}{2}}\left( \frac{1}{%
\left\vert B_{r}\left( x\right) \right\vert }\int\limits_{B_{r}\left(
x\right) }\frac{1}{w\left( y\right) }~dy\right) ^{\frac{1}{2}}<\infty ,
\end{equation*}%
where $\left\vert E\right\vert $ denotes the Lebesgue measure of a
measurable set $E\subset \mathbb{R}^{n}$.

Given $\mathbf{A}\in \mathcal{F}_{n}\left( \lambda ,\Lambda \right) $ and $%
w\in A_{2}$ there is an associated weighted elliptic operator $L_{\mathbf{A}%
,w}=L_{w}:\mathcal{D}\left( L_{w}\right) \rightarrow L^{2}\left( w\right) $
given by%
\begin{equation*}
L_{w}u\left( x\right) =-\frac{1}{w\left( x\right) }\func{div}w\left(
x\right) \mathbf{A}\left( x\right) \nabla u\left( x\right)
\end{equation*}%
in the weak sense. Such operators are closed and sectorial, and so they have
a holomorphic functional calculus which, in particular, enables the
definition of fractional powers $\left( L_{w}\right) ^{\sigma }$.

\begin{theorem}
\label{theorem-weighted}Let $0<\sigma \leq 1$, for every $\mathbf{A}\in 
\mathcal{F}_{n}\left( \lambda ,\Lambda \right) $ and $w\in A_{2}$ there
exists a constant $M=M\left( \lambda ,\Lambda ,\left[ w\right]
_{A_{2}},\sigma \right) $ such that if $\Omega $ is an open set in $\mathbb{R%
}^{n}$, $u\in \mathcal{D}\left( L^{\sigma }\right) \subset L^{2}\left(
\Omega ,w\right) $ and $\left( L_{w}\right) ^{\sigma }u=0$ in an open set $%
\Omega ^{\prime }\subseteq \Omega $, we have that whenever $B_{2r}\left(
x\right) \subset \Omega ^{\prime }$ 
\begin{equation*}
\max_{B_{r}\left( x_{0}\right) }\left\vert u\left( x\right) \right\vert \leq 
\frac{M}{w\left( B_{2r}\left( x_{0}\right) \right) }\left\Vert u\right\Vert
_{L^{2}\left( w,\Omega \right) }^{2}.
\end{equation*}%
Moreover, if $u$ is nonnegative, then whenever $B_{2r}\left( x\right)
\subset \Omega ^{\prime }$ 
\begin{equation*}
\sup_{B_{r}\left( x\right) }u\leq M\inf_{B_{r}\left( x\right) }u.
\end{equation*}%
In particular, solutions of $\left( L_{w}\right) ^{\sigma }u=0$ in $\Omega
^{\prime }$ are locally H\"{o}lder continuous in $\Omega ^{\prime }$.
\end{theorem}

Another consequence of the extension technique applied to weighted elliptic
operators and the Fabes-Kenig-Serapioni boundary Harnack we also obtain, in
the same way as in \cite{caffarelli-silvestre-07} (Theorem 5.3), boundary H%
\"{o}lder continuity for solutions to fractional powers of $L_{w}$. We
present this result to showcase tha applicability of our extension theorem.

\begin{theorem}
\label{theorem-BH}Let $0<\sigma \leq 1$, for every $\mathbf{A}\in \mathcal{F}%
_{n}\left( \lambda ,\Lambda \right) $ and $w\in A_{2}$. Suppose $u\in 
\mathcal{D}\left( L^{\sigma }\right) $ is a function on $\mathbb{R}^{n}$
such that $\left( L_{w}\right) ^{\sigma }u=0$ in a domain $\Omega $, and
suppose that for some $x_{0}\in \Omega $, $u=0$ on $B_{1}\left( x_{0}\right)
\backslash \Omega $ where $\partial \Omega \bigcap B_{1}\left( x_{0}\right) $
is given by a Lipschitz graph with constant less than $1$. Then there exist
constants $M>0$ and $0<\alpha <1$ depending on $\lambda ,\Lambda ,\left[ w%
\right] _{A_{2}}$, and $\sigma $ such that for all $0<\rho <\frac{1}{2}$%
\begin{equation*}
\sup_{\Omega \bigcap B_{\rho }\left( x_{0}\right) }u-\inf_{\Omega \bigcap
B_{\rho }\left( x_{0}\right) }u\leq M\left( \frac{1}{w\left( B_{\frac{1}{2}%
}\left( x_{0}\right) \right) }\int_{B_{\frac{1}{2}}\left( x_{0}\right)
}u^{2}dw\right) ^{\frac{1}{2}}\rho ^{\alpha }.
\end{equation*}
\end{theorem}

The second type of operators we consider illustrates that the reach of our
extension theorem. For subelliptic operators controlled by certain diagonal
matrices, we establish a Harnack estimate for nonnegative solutions to the
square root of such operators. The innovation of this application lays on
the non-isotropic nature of the operators, for which the eigenvalues are
allowed to vanish to different finite orders. The set where an eigenvalue
vanishes may have codimension as small as one.

\begin{remark}
In the subelliptic case we only treat the square root operator $L^{1/2}$.
When $0<\sigma \neq \frac{1}{2}<1$ the resulting equation (\ref{equation})
becomes \emph{weighted subelliptic} with an $A_{2}$ weight depending on the
new variable. We conjecture that the theory developed by Sawyer and Wheeden
in \cite{sawyer-wheeden-06} for subelliptic operators may be extended to
include weighted subelliptic operators with $A_{2}$ weights, and, in such
case, the conclusions of Theorem \ref{theorem-noniso} would hold for all
powers $0<\sigma <1$.
\end{remark}

The geometry for which the Harnack's estimates hold is determined by the
operator's principal terms. We now call up some relevant definitions. A
vector field $X=\mathbf{v}\left( x\right) \cdot \nabla $ defined in an open
set $\Omega \subset \mathbb{R}^{n}$ is said to be \emph{subunit} with
respect to a nonnegative quadratic form $\mathcal{Q}$ in $\Omega $ if 
\begin{equation*}
\left( \mathbf{v}\left( x\right) \cdot \xi \right) ^{2}\leq \mathcal{Q}%
\left( x,\xi \right) \qquad \text{for all }x\in \Omega ,~\xi \in \mathbb{R}%
^{n}.
\end{equation*}%
Given a nonnegative matrix $\mathbf{B}\left( x\right) $, or a system of
vector fields 
\begin{equation*}
\mathbb{X}=\left\{ X_{i}=\mathbf{v}^{i}\cdot \nabla \right\} _{i=1}^{m}
\end{equation*}%
in $\Omega $, they determine quadratic forms $\mathcal{Q}_{\mathbf{B}}\left(
x,\xi \right) =\xi \cdot \mathbf{B}\left( x\right) \xi $ and $\mathcal{Q}_{%
\mathbb{X}}\left( x,\xi \right) =\sum_{i=1}^{m}\left( \mathbf{v}^{i}\cdot
\xi \right) ^{2}$; a vector field $X$ is said to be subunit with respect to
either $\mathbf{B}$ or $\mathbb{X}$ if it is subunit with respect to the
corresponding quadratic form. A Lipschitz curve $\gamma \left( t\right) $ in 
$\Omega $ is said to be subunit with respect to $\mathcal{Q}$ if $\gamma
^{\prime }\left( t\right) $ is a subunit vector field with respect to $%
\mathcal{Q}$. Given a quadratic form $\mathcal{Q}$ in $\Omega $, the subunit
metric associated to $\mathcal{Q}$ is given by%
\begin{equation*}
\delta \left( x,y\right) =\inf \left\{ r>0:\gamma \left( 0\right) =x,~\gamma
\left( r\right) =y,~\gamma \text{ is Lipschitz and subunit}\right\} .
\end{equation*}%
This metric was introduced by Fefferman and Phong in \cite%
{fefferman-phong-81} where they characterize subellipticity for operators
with smooth coefficients.

The following theorem is a special case of a more general result proven in
Section \ref{section-noni}, in which drift terms and zero order terms are
considered. We present this simplified version first for clarity. The
Harnack's inequality for subelliptic operators with rough coefficients
utilized here was established in \cite{sawyer-wheeden-06}.

\begin{theorem}
\label{theorem-noniso}Let $\Omega $ be an open set in $\mathbb{R}^{n}$ and
let $a_{1},\dots ,a_{n}$ be nonnegative Lipschitz functions in $\Omega $
such that for each $x_{0}\in \Omega $ there exists a neighbourhood $\mathcal{%
N}$ of $x_{0}$ in $\Omega $ and a permutation $\tau =\tau _{x_{0}}$ of the
set $\left\{ 1,\cdots ,n\right\} $ so that for $\tau \left( y\right)
:=\left( y_{1},\dots ,y_{n}\right) =\left( x_{\tau \left( 1\right) },\dots
,x_{\tau \left( n\right) }\right) $ and $\tilde{a}_{j}\left( y\right)
=a_{\tau \left( j\right) }\left( \tau ^{-1}\left( y\right) \right) $, we
have that $\tilde{a}_{1}\approx 1$ in $\mathcal{N}$, and%
\begin{equation*}
\tilde{a}_{j}\left( y\right) =\tilde{a}_{j}\left( y_{1},\dots
,y_{j-1}\right) \approx \left( y_{1}^{2}+\dots +y_{j-1}^{2}\right) ^{\frac{%
k_{j}}{2}}\qquad \text{in }\mathcal{N}
\end{equation*}%
for some nonnegative integers $k_{j}$, $j=2,\cdots ,n$. Let $\mathbf{B}%
\left( x\right) $ be an $n\times n$ measurable matrix function in $\Omega $
such that for some constants $0<c_{\mathbf{B}}\leq C_{\mathbf{B}}<\infty $%
\begin{equation*}
c_{\mathbf{B}}\sum_{j=1}^{n}a_{j}^{2}\left( x\right) \xi _{j}^{2}\leq \xi
^{\prime }\mathbf{B}\left( x\right) \xi \leq C_{\mathbf{B}%
}\sum_{j=1}^{n}a_{j}^{2}\left( x\right) \xi _{j}^{2}
\end{equation*}%
for all $x\in \Omega $ and $\xi \in \mathbb{R}^{n}$. Let $\mathcal{L}=-\func{%
div}\mathbf{B}\left( x\right) \nabla $ in $\Omega $, suppose that $u\in 
\mathcal{D}\left( \mathcal{L}^{\frac{1}{2}}\right) $, and that $\mathcal{L}^{%
\frac{1}{2}}u=0$ in an open set $\Omega ^{\prime }\Subset \Omega $. Then
there exist a constant $C_{H}>0$ such that for every ball $B_{2r}\left(
x_{0}\right) \subset \Omega ^{\prime }$, $u$ satisfies%
\begin{equation*}
\limfunc{ess}\sup_{B_{r}\left( x_{0}\right) }u\leq C_{H}\left( \frac{1}{%
\left\vert B_{r}\left( x_{0}\right) \right\vert }\int_{B_{r}\left(
x_{0}\right) }\left\vert u\right\vert ^{2}dx\right) ^{\frac{1}{2}}.
\end{equation*}%
where the balls $B_{r}$ are the subunit metric balls of the metric induced
by the system of vector fields $\left\{ a_{j}\left( x\right) \frac{\partial 
}{\partial x_{j}}\right\} _{j=1}^{n}$. Moreover, if $u$ is nonnegative, then
the Harnack's estimate holds:%
\begin{equation*}
\limfunc{ess}\sup_{B_{r}}u\leq C_{H}~\limfunc{ess}\inf_{B_{r}}u.
\end{equation*}
\end{theorem}

Some specific examples of operators included in Theorem \ref{theorem-noniso}
are: $L_{2}=\frac{\partial }{\partial x^{2}}+\left\vert x\right\vert ^{k_{1}}%
\frac{\partial }{\partial y^{2}}$, and $L_{3}=\frac{\partial }{\partial x^{2}%
}+\left\vert x\right\vert ^{k_{2}}\frac{\partial }{\partial y^{2}}+\left(
\left\vert x\right\vert ^{k_{3}}+\left\vert y\right\vert ^{k_{4}}\right) 
\frac{\partial }{\partial z^{2}}$, where $k_{1},k_{2}\geq 1$.

Our final application is in the nondivergence case for operator with
coefficients with minimal regularity. We obtain a priori estimates for
solutions of $\mathfrak{L}_{\mathbf{A}}^{\sigma }u=0$ for $\sigma $ in a
range depending on $p$, we show that such solutions are in $C^{1,\alpha }$
for all $0<\alpha <1$.

Given $0<\lambda \leq \Lambda <\infty $, and $\mathbf{A}\in \mathcal{F}%
_{n}\left( \lambda ,\Lambda \right) $ we denote by $\mathfrak{L}_{\mathbf{A}%
} $ the nondivergence form operator%
\begin{equation}
\mathfrak{L}_{\mathbf{A}}u=-a^{ij}\left( x\right) \frac{\partial ^{2}u}{%
\partial x_{i}\partial x_{j}}=-\limfunc{trace}\left( \mathbf{A}D^{2}u\right)
,  \label{L-nond}
\end{equation}%
where we adopt the Einstein summation convention. The operator $\mathfrak{L}%
_{\mathbf{A}}$ acts on $W^{2,p}\left( \mathbb{R}^{n}\right) $ $1\leq p\leq
\infty $, i.e. $\mathfrak{L}_{\mathbf{A}}:W^{2,p}\left( \mathbb{R}%
^{n}\right) \rightarrow L^{p}\left( \mathbb{R}^{n}\right) $. We define the
angle $\omega $ by%
\begin{equation}
\omega =\sup_{x\in \mathbb{R}^{n}}\left\{ \arg \left( a^{ij}\left( x\right)
\eta _{i}\overline{\eta }_{j}\right) :\eta \in \mathbb{C}^{n}\right\}
=\arctan \left( \frac{\Lambda }{\lambda }\right) <\frac{\pi }{2}.
\label{omega}
\end{equation}

We now recall the definition of $BMO\left( \mathbb{R}^{n}\right) $, a
measurable function $f$ is in $BMO\left( \mathbb{R}^{n}\right) $ if%
\begin{equation*}
\left\Vert f\right\Vert _{\ast }=\sup_{B\subset \mathbb{R}^{n}}\frac{1}{%
\left\vert B\right\vert }\int_{B}\left\vert f\left( x\right)
-f_{B}\right\vert dx<\infty
\end{equation*}%
where $B$ ranges over all balls in $\mathbb{R}^{n}$, and $f_{B}=\frac{1}{%
\left\vert B\right\vert }\int_{B}f\left( x\right) dx$. For a vector or
matrix function, its BMO norm is defined as the maximum of the BMO norms of
each of its components.

We obtain the following a priori estimate for solutions of the fractional
operator.

\begin{theorem}
\label{theorem-nond}For every dimension $n\geq 1$, and constants $0<\lambda
\leq \Lambda <\infty $, there exists $\varepsilon _{0}=\varepsilon
_{0}\left( n,\lambda ,\Lambda \right) >0$ such that if $\mathbf{A}\in 
\mathcal{F}_{n}\left( \lambda ,\Lambda \right) $ and $\left\Vert \mathbf{A}%
\right\Vert _{\ast }<\varepsilon _{0}$ then if $0<\sigma <\frac{p}{p+1}$ and 
$u\in \mathcal{D}\left( \mathfrak{L}_{\mathbf{A}}\right) \subset \mathcal{D}%
\left( \mathfrak{L}_{\mathbf{A}}^{\sigma }\right) \subset L^{p}\left( 
\mathbb{R}^{n}\right) $ for some $1<p<\infty $, is such that $\mathfrak{L}_{%
\mathbf{A}}^{\sigma }u=0$ in a nonempty open set $\Omega \subset \mathbb{R}%
^{n}$, then $u\in C^{1,\alpha }\left( \Omega \right) $ for all $0<\alpha <1$.
\end{theorem}

\section{The extension theorem\label{section-main}}

In this section we prove the extension theorem for closed operators on a
Banach space. The existence and properties of the functional calculus for
closed operators (not necessarily sectorial) may be found in \cite{haase06,
bandara-mcintosh10}.

A result similar to the following theorem was first obtained in \cite%
{torrea-stinga-10} for self-adjoint second order differential operators in a
Hilbert space. 

\begin{theorem}
\label{theorem-ABGR}Let $\mathcal{X}$ be a Banach space and let $T$ be a
densely defined closed operator on $\mathcal{D}\left( T\right) \subset 
\mathcal{X}$. For $0<\sigma <1$, let $x_{0}\in \mathcal{D}\left( T^{\sigma
}\right) \subset \mathcal{X}$. The function $x:\left[ 0,\infty \right)
\mapsto \mathcal{X}$ defined by $x\left( 0\right) =x_{0}$ and 
\begin{equation*}
x\left( y\right) =\frac{1}{\Gamma \left( \sigma \right) }\int_{0}^{\infty
}e^{-tT}\left( \left( tT\right) ^{\sigma }x_{0}\right) e^{-\frac{y^{2}}{4t}}%
\frac{dt}{t},
\end{equation*}%
satisfies $x\left( y\right) \in \mathcal{D}\left( T\right) $ for all $y>0$
and it is a solution of the initial value problem: $x\left( 0\right) =x_{0}$%
, and%
\begin{equation}
-Tx\left( y\right) +\frac{1-2\sigma }{y}x^{\prime }\left( y\right)
+x^{\prime \prime }\left( y\right) =0\qquad \qquad \text{in }\left( 0,\infty
\right)  \label{equation}
\end{equation}%
in the sense that $x\left( y\right) \rightarrow x_{0}$ in $\mathcal{X}$ as $%
y\rightarrow 0^{+}$ and the above differential equation holds in $\mathcal{X}
$. Moreover, we have that%
\begin{eqnarray}
\lim_{y\rightarrow 0^{+}}\frac{x\left( y\right) -x_{0}}{y^{2\sigma }} &=&%
\frac{\Gamma \left( -\sigma \right) }{4^{\sigma }\Gamma \left( \sigma
\right) }T^{\sigma }x_{0}=\frac{1}{2\sigma }\lim_{y\rightarrow
0^{+}}y^{1-2\sigma }x^{\prime }\left( y\right) ;  \label{limits} \\
\frac{1}{2\sigma }\lim_{y\rightarrow 0^{+}}y^{2-2\sigma }x^{\prime \prime
}\left( y\right) &=&\frac{2\sigma -1}{4^{\sigma }}\frac{\Gamma \left(
-\sigma \right) }{\Gamma \left( \sigma \right) }T^{\sigma }x_{0};
\label{limits2}
\end{eqnarray}%
the following Poisson formula holds for $x$:%
\begin{equation}
x\left( y\right) =\frac{y^{2\sigma }}{4^{\sigma }\Gamma \left( \sigma
\right) }\int_{0}^{\infty }e^{-tT}e^{-\frac{y^{2}}{4t}}x_{0}\frac{dt}{%
t^{1+\sigma }}=\frac{1}{\Gamma \left( \sigma \right) }\int_{0}^{\infty }e^{-%
\frac{y^{2}}{4r}T}x_{0}\frac{dr}{r^{1-\sigma }};  \label{Poisson}
\end{equation}%
and for all $n\geq 0$ we have the bounds%
\begin{equation}
\left\Vert \frac{d}{dy^{n}}x\left( y\right) \right\Vert \leq \frac{C}{y^{n}}%
\left\Vert x_{0}\right\Vert ,\qquad \text{for all }y>0,
\label{der-estimates}
\end{equation}%
and%
\begin{equation}
\left\Vert T^{\sigma }x\left( y\right) \right\Vert \leq C\left\Vert
T^{\sigma }x_{0}\right\Vert .  \label{frac-estimates}
\end{equation}%
Finally, if $x_{0}\in \mathcal{D}\left( T\right) $, we also have that 
\begin{equation}
\lim_{y\rightarrow 0^{+}}Tx\left( y\right) =Tx_{0}.  \label{limits3}
\end{equation}
\end{theorem}

\begin{proof}
First note that since $e^{-tz}\left( tz\right) ^{\sigma }\in H_{0}^{\infty
}\left( \Sigma _{\pi /2-\varepsilon }\right) $ (see the Appendix for details
on the functional calculus) for any fixed $0<\varepsilon <\pi /2$, we have
that $x\left( y\right) =\psi _{y}\left( T\right) x_{0}$, where%
\begin{equation}
\psi _{y}\left( z\right) =\frac{1}{\Gamma \left( \sigma \right) }%
\int_{0}^{\infty }e^{-tz}\left( tz\right) ^{\sigma }e^{-\frac{y^{2}}{4t}}%
\frac{dt}{t}\in H^{\infty }\left( \Sigma _{\pi /2-\varepsilon }\right) \quad 
\text{uniformly in }y>0  \label{eq-psi}
\end{equation}%
so by (\ref{eqn:L2-holo-bd}) it follows that $\left\Vert x\left( y\right)
\right\Vert \leq C\left\Vert x_{0}\right\Vert $ where $C=\left\Vert \psi
\right\Vert _{L^{\infty }\left( \Sigma _{\pi /2-\varepsilon }\right) }$ and
so $x\left( y\right) $ is well defined in $\mathcal{X}$. Since 
\begin{equation*}
\lim_{y\rightarrow 0^{+}}\psi _{y}\left( z\right) =\frac{1}{\Gamma \left(
\sigma \right) }\int_{0}^{\infty }e^{-tz}\left( tz\right) ^{\sigma }\frac{dt%
}{t}=\frac{1}{\Gamma \left( \sigma \right) }\int_{0}^{\infty
}e^{-s}s^{\sigma }\frac{ds}{s}=1,
\end{equation*}%
we have that $x\left( y\right) =\psi _{y}\left( T\right) x_{0}\rightarrow
x_{0}$ in $\mathcal{X}$ as $y\rightarrow 0^{+}$. Also, since $e^{-tT}:%
\mathcal{X}\rightarrow \mathcal{D}\left( T\right) $ it follows that $x\left(
y\right) \in \mathcal{D}\left( T\right) $ for all $y>0$. Moreover, from the
estimates 
\begin{equation*}
\left\vert \frac{d^{m}}{dy^{m}}e^{-\frac{y^{2}}{4t}}\right\vert \leq \frac{%
C_{m}}{y^{m}}\qquad m=0,1,2,\dots 
\end{equation*}%
with $C_{m}$ independent of $t$, it follows that derivatives of $x\left(
y\right) $ may be computed by taking derivatives inside the integral:%
\begin{equation}
\frac{d^{n}}{dy^{n}}x\left( y\right) =\frac{1}{\Gamma \left( \sigma \right) }%
\int_{0}^{\infty }e^{-tT}\left( \left( tT\right) ^{\sigma }x_{0}\right)
\left( \frac{d^{n}}{dy^{n}}e^{-\frac{y^{2}}{4t}}\right) \frac{dt}{t},
\label{eq-firstd}
\end{equation}%
for all $n=0,1,2,\dots $. In particular, it follows as before by the
boundedness of the functional calculus (\ref{eqn:L2-holo-bd}) that%
\begin{equation*}
\left\Vert \frac{d^{n}}{dy^{n}}x\left( y\right) \right\Vert \leq \frac{C_{n}%
}{y^{n}}\left\Vert x_{0}\right\Vert \qquad \text{for all }n\geq 0.
\end{equation*}

This proves (\ref{der-estimates}). We remark that for an explicit resolution
of the fractional powers $T^{\sigma }$ we may use (\ref{eq-fract}). Now,
since%
\begin{eqnarray*}
\frac{1-2\sigma }{y}\left( \frac{d}{dy}e^{-\frac{y^{2}}{4t}}\right) +\left( 
\frac{d^{2}}{dy^{2}}e^{-\frac{y^{2}}{4t}}\right)  &=&\left( -\frac{1-2\sigma 
}{y}\frac{y}{2t}-\frac{1}{2t}+\frac{y^{2}}{4t^{2}}\right) e^{-\frac{y^{2}}{4t%
}} \\
&=&\left( \frac{\sigma -1}{t}+\frac{y^{2}}{4t^{2}}\right) e^{-\frac{y^{2}}{4t%
}}
\end{eqnarray*}%
we have%
\begin{eqnarray*}
&&\frac{1-2\sigma }{y}x^{\prime }\left( y\right) +x^{\prime \prime }\left(
y\right)  \\
&=&\frac{1}{\Gamma \left( \sigma \right) }\int_{0}^{\infty }e^{-tT}\left(
T^{\sigma }x_{0}\right) t^{\sigma -1}\left( \frac{\sigma -1}{t}+\frac{y^{2}}{%
4t^{2}}\right) e^{-\frac{y^{2}}{4t}}dt \\
&=&\frac{1}{\Gamma \left( \sigma \right) }\int_{0}^{\infty }e^{-tT}\left(
T^{\sigma }x_{0}\right) \left( \frac{d}{dt}t^{\sigma -1}e^{-\frac{y^{2}}{4t}%
}\right) ~dt \\
&=&\left. \frac{1}{\Gamma \left( \sigma \right) }e^{-tT}\left( T^{\sigma
}x_{0}\right) t^{\sigma -1}e^{-\frac{y^{2}}{4t}}\right\vert _{t=0}^{\infty }
\\
&&+\frac{1}{\Gamma \left( \sigma \right) }\int_{0}^{\infty }\left( -\frac{d}{%
dt}e^{-tT}\right) \left( T^{\sigma }x_{0}\right) \left( t^{\sigma -1}e^{-%
\frac{y^{2}}{4t}}\right) dt \\
&=&\frac{1}{\Gamma \left( \sigma \right) }\int_{0}^{\infty }Te^{-tT}\left(
\left( tT\right) ^{\sigma }x_{0}\right) e^{-\frac{y^{2}}{4t}}\frac{dt}{t} \\
&=&Tx\left( y\right) .
\end{eqnarray*}%
This proves that $x\left( y\right) $ satisfies equation (\ref{equation}).
Now, since 
\begin{equation*}
\frac{1}{\Gamma \left( \sigma \right) }\int_{0}^{\infty }e^{-tz}\left(
tz\right) ^{\sigma }\frac{dt}{t}=1,\qquad z\in \Sigma _{\pi /2},
\end{equation*}%
we can write%
\begin{equation*}
x_{0}=\frac{1}{\Gamma \left( \sigma \right) }\int_{0}^{\infty }e^{-tT}\left(
tT\right) ^{\sigma }x_{0}\frac{dt}{t},
\end{equation*}%
and therefore%
\begin{eqnarray*}
\lim_{y\rightarrow 0^{+}}\frac{x\left( y\right) -x_{0}}{y^{2\sigma }}
&=&\lim_{y\rightarrow 0^{+}}\frac{1}{\Gamma \left( \sigma \right) }%
\int_{0}^{\infty }e^{-tT}\left( \left( tT\right) ^{\sigma }x_{0}\right)
\left( \frac{e^{-\frac{y^{2}}{4t}}-1}{y^{2\sigma }}\right) \frac{dt}{t} \\
&=&\lim_{y\rightarrow 0^{+}}\frac{1}{4^{\sigma }\Gamma \left( \sigma \right) 
}\int_{0}^{\infty }e^{-tT}\left( T^{\sigma }x_{0}\right) \left( \frac{e^{-%
\frac{y^{2}}{4t}}-1}{\left( \frac{y^{2}}{4t}\right) ^{\sigma }}\right) \frac{%
dt}{t}.
\end{eqnarray*}%
Performing the change of variables $s=y^{2}/\left( 4t\right) $,%
\begin{equation*}
\lim_{y\rightarrow 0^{+}}\frac{x\left( y\right) -x_{0}}{y^{2\sigma }}%
=\lim_{y\rightarrow 0^{+}}\frac{1}{4^{\sigma }\Gamma \left( \sigma \right) }%
\int_{0}^{\infty }e^{-\frac{y^{2}}{4s}T}\left( T^{\sigma }x_{0}\right)
\left( \frac{e^{-s}-1}{s^{\sigma }}\right) \frac{ds}{s}.
\end{equation*}%
Since for any fixed $\varepsilon >0$ $e^{-\frac{y^{2}}{4s}T}\rightarrow I$
as $y\rightarrow 0^{+}$, strongly and uniformly on $s\in \left[ \varepsilon
,\infty \right) $, and $\left\Vert e^{-\frac{y^{2}}{4s}T}\right\Vert \leq 1$
for all $s,y>0$, we conclude that%
\begin{eqnarray*}
\lim_{y\rightarrow 0^{+}}\frac{x\left( y\right) -x_{0}}{y^{2\sigma }} &=&%
\frac{1}{4^{\sigma }\Gamma \left( \sigma \right) }T^{\sigma
}x_{0}\int_{0}^{\infty }\left( \frac{e^{-s}-1}{s^{\sigma }}\right) \frac{ds}{%
s} \\
&=&\frac{\Gamma \left( -\sigma \right) }{4^{\sigma }\Gamma \left( \sigma
\right) }T^{\sigma }x_{0}.
\end{eqnarray*}%
This establishes the first equality in (\ref{limits}). Similarly, from (\ref%
{eq-firstd}) we have%
\begin{eqnarray*}
\frac{1}{2\sigma }\lim_{y\rightarrow 0^{+}}y^{1-2\sigma }x^{\prime }\left(
y\right)  &=&-\frac{1}{2\sigma }\frac{1}{\Gamma \left( \sigma \right) }%
\lim_{y\rightarrow 0^{+}}\int_{0}^{\infty }e^{-tT}\left( \left( tT\right)
^{\sigma }x_{0}\right) \frac{y^{2-2\sigma }}{2t}e^{-\frac{y^{2}}{4t}}\frac{dt%
}{t} \\
&=&-\frac{1}{\sigma }\frac{1}{4^{\sigma }\Gamma \left( \sigma \right) }%
\lim_{y\rightarrow 0^{+}}\int_{0}^{\infty }e^{-tT}\left( T^{\sigma
}x_{0}\right) \left( \frac{y^{2}}{4t}\right) ^{1-\sigma }e^{-\frac{y^{2}}{4t}%
}\frac{dt}{t} \\
&=&-\frac{1}{\sigma }\frac{1}{4^{\sigma }\Gamma \left( \sigma \right) }%
\lim_{y\rightarrow 0^{+}}\int_{0}^{\infty }e^{-\frac{y^{2}}{4s}T}\left(
T^{\sigma }x_{0}\right) s^{1-\sigma }e^{-s}\frac{ds}{s} \\
&=&-\frac{1}{\sigma }\frac{1}{4^{\sigma }\Gamma \left( \sigma \right) }%
T^{\sigma }x_{0}\int_{0}^{\infty }s^{1-\sigma }e^{-s}\frac{ds}{s} \\
&=&-\frac{1}{\sigma }\frac{\Gamma \left( 1-\sigma \right) }{4^{\sigma
}\Gamma \left( \sigma \right) }T^{\sigma }x_{0}=\frac{\Gamma \left( -\sigma
\right) }{4^{\sigma }\Gamma \left( \sigma \right) }T^{\sigma }x_{0},
\end{eqnarray*}%
which establishes the second equality in (\ref{limits}). In this manner we
also obtain%
\begin{eqnarray*}
\frac{1}{2\sigma }y^{2-2\sigma }x^{\prime \prime }\left( y\right)  &=&\frac{%
4^{1-\sigma }}{2\sigma \Gamma \left( \sigma \right) }\int_{0}^{\infty }e^{-%
\frac{y^{2}}{4s}T}\left( T^{\sigma }x_{0}\right) s^{2-\sigma }e^{-s}\frac{ds%
}{s} \\
&&-\frac{4^{1-\sigma }}{4\sigma \Gamma \left( \sigma \right) }%
\int_{0}^{\infty }e^{-\frac{y^{2}}{4s}T}\left( T^{\sigma }x_{0}\right)
s^{1-\sigma }e^{-s}\frac{ds}{s}
\end{eqnarray*}%
which yields (\ref{limits2}) as $y\rightarrow 0^{+}$.

Now, from (\ref{eqn:L2-holo-rep}) and (\ref{eqn:L2-holo-rep-eta}) we have
the representation (all the expressions are valued at $x_{0}$)%
\begin{eqnarray*}
x\left( y\right) &=&\frac{1}{\Gamma \left( \sigma \right) }\int_{0}^{\infty
}e^{-tT}\left( tT\right) ^{\sigma }e^{-\frac{y^{2}}{4t}}\frac{dt}{t} \\
&=&\frac{1}{\Gamma \left( \sigma \right) }\int_{0}^{\infty }\int_{\Gamma
_{\pi /2-\theta }}e^{-zT}\eta (z)\,dze^{-\frac{y^{2}}{4t}}\frac{dt}{t} \\
&=&\frac{1}{\Gamma \left( \sigma \right) }\int_{0}^{\infty }\int_{\Gamma
_{\pi /2-\theta }}e^{-zT}\frac{1}{2\pi i}\int_{\gamma _{\nu }(z)}e^{\zeta
z}e^{-t\zeta }\left( t\zeta \right) ^{\sigma }\,d\zeta \,dz~e^{-\frac{y^{2}}{%
4t}}\frac{dt}{t}.
\end{eqnarray*}%
Since $\zeta \in \gamma _{\nu }\left( z\right) $ does not vanish, and using
Fubini's theorem, we can perform the change of variables $t=\frac{y^{2}}{%
4\zeta \rho }$, so $dt=-\frac{y^{2}}{4\zeta \rho ^{2}}d\rho $, to obtain%
\begin{eqnarray*}
x\left( y\right) &=&\frac{1}{\Gamma \left( \sigma \right) }\int_{\Gamma
_{\pi /2-\theta }}e^{-zT}\frac{1}{2\pi i}\int_{\gamma _{\nu }(z)}e^{\zeta
z}\int_{0}^{\infty }e^{-t\zeta }\left( t\zeta \right) ^{\sigma }\,e^{-\frac{%
y^{2}}{4t}}\frac{dt}{t}d\zeta \,dz~ \\
&=&\frac{1}{\Gamma \left( \sigma \right) }\int_{\Gamma _{\pi /2-\theta
}}e^{-zT}\frac{1}{2\pi i}\int_{\gamma _{\nu }(z)}e^{\zeta z}\int_{\overline{%
\gamma _{\nu }\left( z\right) }}e^{-\frac{y^{2}}{4\rho }}\left( \frac{y^{2}}{%
4\rho }\right) ^{\sigma }\,e^{-\zeta \rho }\frac{d\rho }{\rho }d\zeta \,dz,
\end{eqnarray*}%
where $\overline{\gamma _{\nu }\left( z\right) }$ is the conjugate path to $%
\gamma _{\nu }\left( z\right) $. By homotopy and the Cauchy's integral
formula, the path of integration $\overline{\gamma _{\nu }\left( z\right) }$
may be replaced by $\mathbb{R}^{+}$, so it follows that%
\begin{eqnarray*}
x\left( y\right) &=&\frac{1}{\Gamma \left( \sigma \right) }\int_{\Gamma
_{\pi /2-\theta }}e^{-zT}\frac{1}{2\pi i}\int_{\gamma _{\nu }(z)}e^{\zeta
z}\int_{0}^{\infty }e^{-\frac{y^{2}}{4t}}\left( \frac{y^{2}}{4t}\right)
^{\sigma }\,e^{-\zeta t}\frac{dt}{t}d\zeta \,dz \\
&=&\frac{1}{\Gamma \left( \sigma \right) }\left( \frac{y^{2}}{4}\right)
^{\sigma }\int_{0}^{\infty }\left( \int_{\Gamma _{\pi /2-\theta }}e^{-zT}%
\frac{1}{2\pi i}\int_{\gamma _{\nu }(z)}e^{\zeta z}\,e^{-\zeta t}d\zeta
\,dz\right) e^{-\frac{y^{2}}{4t}}\frac{dt}{t^{1+\sigma }} \\
&=&\frac{1}{\Gamma \left( \sigma \right) }\left( \frac{y^{2}}{4}\right)
^{\sigma }\int_{0}^{\infty }e^{-tT}e^{-\frac{y^{2}}{4t}}\frac{dt}{%
t^{1+\sigma }}
\end{eqnarray*}%
where we again used (\ref{eqn:L2-holo-rep}) and (\ref{eqn:L2-holo-rep-eta}).
This proves the first equality in (\ref{Poisson}). The second equality in (%
\ref{Poisson}) follows upon implementing the change of variables $t=\frac{%
y^{2}}{4s}$ on the second term.

Finally, from the identity $x\left( y\right) =\psi _{y}\left( T\right) x_{0}$
we have that $Tx\left( y\right) ={\varphi }_{y}\left( T\right) x_{0}$ with ${%
\varphi }_{y}\left( z\right) =z\psi _{y}\left( z\right) $. Now, ${\varphi }%
_{y}\left( z\right) \rightarrow z$ locally uniformly in $\Sigma _{\pi
/2-\varepsilon }$, so if $x_{0}\in \mathcal{D}\left( T\right) $ then (\ref%
{limits3}) follows.
\end{proof}

\begin{remark}
The solution $\psi \left( T\right) x_{0}$ with $\psi \left( z\right) $ given
by (\ref{eq-psi}) must be understood in a limit sense in the Banach space $%
\mathcal{X}$. Note that $\psi \left( 0\right) =1,$ so $\psi $ does not
belong to $H_{0}^{\infty }\left( \Sigma _{\pi /2-\varepsilon }\right) $.
Indeed, the change of variables $s=tz$ yields 
\begin{equation*}
\psi \left( z\right) =\frac{1}{\Gamma \left( \sigma \right) }%
\int_{0}^{\infty }e^{-s}s^{\sigma }e^{-\frac{y^{2}z}{4s}}\frac{ds}{s}
\end{equation*}%
from which we can easily see that $\psi \left( z\right) \rightarrow 1$ as $%
z\rightarrow 0$ in $\Sigma _{\pi /2-\varepsilon }$ (and that $\psi \left(
z\right) \rightarrow 0$ as $z\rightarrow \infty $ in $\Sigma _{\pi
/2-\varepsilon }$). Thus, $\psi \left( T\right) $ shall be understood as a
limit of $\psi _{n}\left( T\right) $ where $\psi _{n}\in H_{0}^{\infty
}\left( \Sigma _{\pi /2-\varepsilon }\right) $ and $\psi _{n}\rightarrow
\psi $ uniformly on compact subsets of $\Sigma _{\pi /2-\varepsilon }$. For
example, it would suffice to take $\psi _{n}\left( z\right) =\psi \left(
z\right) -e^{-nz}$.
\end{remark}

\begin{remark}
Note that the first Poisson formula in (\ref{Poisson}) says that the
solution $x\left( y\right) $ is $x\left( y\right) =\Psi _{y}\left( T\right)
x_{0}$ where the operator $\Psi _{y}\left( T\right) $ is given by%
\begin{equation*}
\Psi _{y}\left( T\right) =\frac{y^{2\sigma }}{4^{\sigma }\Gamma \left(
\sigma \right) }\int_{0}^{\infty }e^{-tT}e^{-\frac{y^{2}}{4t}}\frac{dt}{%
t^{1+\sigma }}.
\end{equation*}%
This operator is indeed the Laplace transform $\mathcal{L}\left(
g_{y}\right) \left( z\right) $ of the function $g_{y}\left( t\right) =\frac{%
y^{2\sigma }}{4^{\sigma }\Gamma \left( \sigma \right) }t^{-\left( 1+\sigma
\right) }e^{-\frac{y^{2}}{4t}}\in L^{1}\left( \left[ 0,\infty \right)
\right) $. The mapping $g\rightarrow \mathcal{L}\left( g\right) $ is called
the \emph{Phillips calculus} for $T$ (see 3.3 in \cite{haase06}).
\end{remark}

\section{Applications to Differential Operators\label{section-applications}}

In this section we provide some applications of the general extension
Theorem \ref{theorem-ABGR} to three different types of differential
operators to illustrate its versatility.

\subsection{Non-symmetric weighted elliptic operators\label%
{section-weighted-elliptic}}

Let $\Omega $ be an open subset of $\mathbb{R}^{n}$. For indexes $1\leq
p\leq \infty $, and a locally integrable nonnegative function $w$ we denote
by $L^{p}\left( \Omega ,w\right) $ the space of measurable functions $f$ on $%
\Omega $ such that $\left\vert f\right\vert ^{p}$ is integrable with respect
to the measure $dw=w\left( x\right) dx$. Given a weight $w\in A_{2}\left(
\Omega \right) $, the weighted Sobolev spaces $\mathcal{H}^{1}\left( \Omega
,w\right) $ consists of those $L^{2}\left( \Omega ,w\right) $ functions $f$
such that $\left\vert \nabla f\right\vert \in L^{2}\left( \Omega ,w\right) $%
. This space is a Hilbert space with inner product%
\begin{equation}
\left\langle u,v\right\rangle _{\mathcal{H}^{1}\left( \Omega ,w\right)
}=\int_{\Omega }uv~dw+\int_{\Omega }\nabla u\cdot \nabla v~dw.
\label{eq-H1w-inner}
\end{equation}%
See \cite{miller82, fabes-kenig-serapioni82} for more details about these
weighted Sobolev spaces. We also adopt the conventions $L^{2}\left( w\right)
=L^{2}\left( \mathbb{R}^{n},w\right) $ and $\mathcal{H}^{1}\left( w\right) =%
\mathcal{H}^{1}\left( \mathbb{R}^{n},w\right) $. In what follows we will
take $\Omega =\mathbb{R}^{n}$ for simplicity, but all the definitions and
properties below can specialized to any subdomain $\Omega \subset \mathbb{R}%
^{n}$.

Recall that given $0<\lambda \leq \Lambda <\infty ,$ $\mathcal{F}_{n}\left(
\lambda ,\Lambda \right) $ denotes the set of real valued $n\times n$
matrices which eigenvalues lie in the interval $\left[ \lambda ,\Lambda %
\right] $. Given any $\mathbf{A}\in \mathcal{F}_{n}\left( \lambda ,\Lambda
\right) $ we can more define a bilinear form $\mathcal{E}_{\mathbf{A},w}$ on 
$\mathcal{H}^{1}\left( w\right) \subset L^{2}\left( w\right) $ by%
\begin{equation*}
\mathcal{E}_{\mathbf{A},w}\left( u,v\right) =\left\langle \mathbf{A}\left(
x\right) \nabla u,\nabla v\right\rangle _{w}=\int_{\mathbb{R}^{n}}\mathbf{A}%
\left( x\right) \nabla u\left( x\right) \cdot \nabla {v}\left( x\right) ~dw.
\end{equation*}%
We will check that $\mathcal{E}_{\mathbf{A},w}$ satisfies conditions (\ref%
{E1-lb}), (\ref{E2-sector}), and (\ref{E3-complete}) from Definition \ref%
{def-Dirichlet}. By ellipticity (\ref{ellipticity}) it is clear that $%
\mathcal{E}_{\mathbf{A},w}$ is nonnegative, and, moreover, $\mathbf{A}\left(
x\right) \xi \cdot \xi $ is an inner product in $\mathbb{R}^{n}$, hence by
Cauchy-Schwarz $\left\vert \mathcal{E}_{\mathbf{A},w}\left( u,v\right)
\right\vert \leq \mathcal{E}_{\mathbf{A},w}\left( u,u\right) ^{1/2}\mathcal{E%
}_{\mathbf{A},w}\left( v,v\right) ^{1/2}$ for all $u,v\in \mathcal{H}%
^{1}\left( w\right) $. Thus $\mathcal{E}_{\mathbf{A},w}$ satisfies
conditions (\ref{E1-lb}) and (\ref{E2-sector}). On the other hand, since the
domain of $\mathcal{E}_{\mathbf{A},w}$ is $\mathcal{H}^{1}\left( w\right) $,
a Hilbert space with inner product (\ref{eq-H1w-inner}), which is equivalent
to 
\begin{equation*}
\mathcal{E}_{\left( \mathbf{A},w\right) ,\alpha }^{\left( s\right) }\left(
u,v\right) =\frac{1}{2}\left( \mathcal{E}_{\left( \mathbf{A},w\right)
}^{\left( s\right) }\left( u,v\right) +\mathcal{E}_{\left( \mathbf{A}%
,w\right) }^{\left( s\right) }\left( v,u\right) \right) +\alpha \left\langle
u,v\right\rangle _{w}
\end{equation*}%
for all $\alpha >0$, we have that $\mathcal{E}_{\mathbf{A},w}$ also
satisfies (\ref{E3-complete}). Hence, $\mathcal{E}_{\mathbf{A},w}$ is a
nonnegative bilinear form satisfying the conditions of Theorem \ref%
{theorem-Dirichlet}, the associated operator $L_{w}=L_{\mathbf{A},w}$ is
closed and it has domain $\mathcal{D}\left( L_{w}\right) $ which is dense in 
$L^{2}\left( w\right) $, and for all $u\in \mathcal{D}\left( L_{w}\right) $
we have that $f\left( x\right) =L_{w}u\left( x\right) \in L^{2}\left(
w\right) $ satisfies 
\begin{equation*}
\int_{\mathbb{R}^{n}}\mathbf{A}\left( x\right) \nabla u\left( x\right) \cdot
\nabla {\varphi }\left( x\right) ~dw=\int_{\mathbb{R}^{n}}f\left( x\right) ~{%
\varphi }\left( x\right) ~dw,\qquad \text{for all }{\varphi }\in \mathcal{H}%
^{1}\left( w\right) .
\end{equation*}%
It is in this weak sense (integrating by parts) that we say that $%
L_{w}u\left( x\right) =-\frac{1}{w}\func{div}w\mathbf{A}\left( x\right)
\nabla u\left( x\right) =f\left( x\right) $. Notice that Theorem \ref%
{theorem-Dirichlet} also yields the adjoint operator $\widehat{L}_{w}$ is
closed with dense domain $\mathcal{D}\left( \widehat{L}_{w}\right) \subset
L^{2}\left( w\right) $ such that%
\begin{equation*}
\int_{\mathbb{R}^{n}}\mathbf{A}\left( x\right) \nabla {\varphi }\left(
x\right) \cdot \nabla {v}\left( x\right) ~dw=\int_{\mathbb{R}^{n}}{\varphi }%
\left( x\right) ~\widehat{L}_{w}v\left( x\right) ~dw,
\end{equation*}%
for all $v\in \mathcal{D}\left( \widehat{L}_{w}\right) ,~{\varphi }\in 
\mathcal{H}^{1}\left( w\right) $. Moreover, by Corollary \ref{coro-sectorial}
it follows that bot $L_{w}$ and $\widehat{L}_{w}$ are sectorial of angle $%
\pi /2-\arctan \left( 1/K\right) $ and they generates strongly continuous
semigroups in $\left[ 0,\infty \right) $ and a contractive holomorphic
semigroups in $\Sigma _{\arctan \left( 1/K\right) }$.

Fabes, Kenig, and Serapioni extended the De Giorgi-Nash-Moser theory to
symmetric weighted elliptic operators $L_{w}$ with $A_{2}$ weights (and with
quasi-conformal weights) \cite{fabes-kenig-serapioni82}. As observed in
their paper, the Moser iteration scheme can more generally be implemented as
far as proper versions of the following a-priori estimates are available:

\begin{enumerate}
\item \bigskip A Caccioppoli inequality. For every $u\in \mathcal{H}%
^{1}\left( w\right) $ and $C_{0}^{\infty }\left( \mathbb{R}^{n}\right) $
function ${\varphi }$%
\begin{equation*}
\int {\varphi }^{2}\left\vert \nabla u\right\vert ^{2}dw\leq C\int
\left\vert {u}\left( x\right) \right\vert ^{2}\left( \left\vert \nabla {%
\varphi }\right\vert ^{2}+{\varphi }^{2}\right) dw~+C\int \left\vert
L_{w}u\right\vert ^{2}~{\varphi }^{2}~dw,
\end{equation*}%
where $C=C\left( \lambda ,\Lambda \right) .$

\item A Sobolev inequality. There exist $p\in \left[ 1,\infty \right) $, $%
k>1 $, and $C=C\left( \lambda ,\Lambda ,w\right) >0$ such that for all balls 
$B_{r}=B_{r}\left( x\right) \subset \mathbb{R}^{n}$ and any $u\in
C_{0}^{\infty }\left( B_{r}\right) $%
\begin{equation*}
\left( \frac{1}{w\left( B_{r}\right) }\int_{B_{r}}\left\vert u\right\vert
^{pk}dw\right) ^{\frac{1}{pk}}\leq Cr\left( \frac{1}{w\left( B_{r}\right) }%
\int_{B_{r}}\left\vert \nabla u\right\vert ^{p}dw\right) ^{\frac{1}{p}}.
\end{equation*}

\item A Poincar\'{e} inequality. There exist $p\in \left[ 1,\infty \right) $%
, $k>1$, and $C=C\left( \lambda ,\Lambda ,w\right) >0$ such that for all
balls $B_{r}=B_{r}\left( x\right) \subset \mathbb{R}^{n}$ and any Lipschitz
function $u$%
\begin{equation*}
\left( \frac{1}{w\left( B_{r}\right) }\int_{B_{r}}\left\vert
u-u_{B_{r}}\right\vert ^{pk}dw\right) ^{\frac{1}{pk}}\leq Cr\left( \frac{1}{%
w\left( B_{r}\right) }\int_{B_{r}}\left\vert \nabla u\right\vert
^{p}dw\right) ^{\frac{1}{p}},
\end{equation*}%
where $u_{B_{r}}=\frac{1}{w\left( B_{r}\right) }\int_{B_{r}}u~dw$.
\end{enumerate}

The De Giorgi-Nash-Moser techniques have been applied in increasing
generality to degenerate elliptic equations, semi-linear equations and fully
nonlinear equations. In most applications, Caccioppoli estimates are an easy
consequence of the definition of weak solutions and integration by parts.
Sobolev and Poincar\'{e} inequalities have been established in a wide
variety of settings, including those of weighted elliptic operators with
symmetric coefficients \cite{fabes-kenig-serapioni82} and subelliptic
equations \cite{sawyer-wheeden-06, rodney10}. As pointed in \cite{KKPT2000},
is was first observed by Morrey \cite{morrey66} (chapter 5) that the De
Giorgi-Nash-Moser theory also holds for solutions to elliptic divergence
form equations without the assumption that the matrix $\mathbf{A}\left(
x\right) $ is symmetric, this fact easily extends to the weighted elliptic
operators defined above. In particular, we have that local boundedness and
Harnack's inequality hold in this setting.

\begin{theorem}[Boundedness of solutions and Harnack's inequality for
weighted elliptic operators]
\label{theorem-FKS} Let $w\in A_{2}$ and let $\Omega \subset \mathbb{R}^{n}$
open. If $u\in \mathcal{H}^{1}\left( \Omega ,w\right) $ is a solution of $%
L_{w}u=0$ in $\Omega $, then there exists a constant $M=M\left( \lambda
,\Lambda ,\left[ w\right] _{A_{2}}\right) >0$ such that for every ball $%
B_{2r}\left( y\right) \subset \Omega $ 
\begin{equation*}
\max_{B_{r}\left( y\right) }\left\vert u\left( x\right) \right\vert \leq M%
\frac{1}{w\left( B_{2r}\left( y\right) \right) }\int_{B_{2r}\left( y\right)
}\left\vert u\right\vert ^{2}dw\left( x\right) .
\end{equation*}%
Moreover, if $u$ is nonnegative then for every ball $B_{2r}\left( y\right)
\subset \Omega $ 
\begin{equation*}
\max_{B_{r}\left( y\right) }u\leq M\min_{B_{r}\left( y\right) }u.
\end{equation*}
\end{theorem}

Note that the case $w\equiv 1$ in the above theorem covers classical
non-symmetric elliptic operators. Now we are ready to apply the main
extension theorem to weighted elliptic operators.

\begin{proof}[Proof of\ Theorem \protect\ref{theorem-weighted}]
We have $u\in \mathcal{D}\left( \left( L_{w}\right) ^{\sigma }\right)
\subset L^{2}\left( \Omega ,w\right) $ is a solution of $\left( L_{w}\right)
^{\sigma }u=0$ in an open set $\Omega \subset \mathbb{R}^{n}$. For $\left(
x,y\right) \in \Omega \times \left( 0,\infty \right) $ we let%
\begin{equation*}
U\left( x,y\right) =\frac{1}{\Gamma \left( \sigma \right) }\int_{0}^{\infty
}e^{-tL_{w}}\left( \left( tL_{w}\right) ^{\sigma }u\left( x\right) \right)
e^{-\frac{y^{2}}{4t}}\frac{dt}{t}.
\end{equation*}%
By Theorem \ref{theorem-ABGR}, with $\mathcal{X}=L^{2}\left( \Omega
,w\right) $, $U\left( x,y\right) $ is well defined and $\left\Vert U\left(
\cdot ,y\right) \right\Vert _{L^{2}\left( w,\Omega \right) }\leq C\left\Vert
u\right\Vert _{L^{2}\left( w,\Omega \right) }$ for all $y>0$. Moreover, $U$
is a solution of the initial value problem%
\begin{eqnarray}
-L_{w}U\left( x,y\right) +\frac{1-2\sigma }{y}\frac{\partial }{\partial y}%
U\left( x,y\right) +\frac{\partial ^{2}}{\partial y^{2}}U\left( x,y\right) 
&=&0  \label{eq-ord} \\
U\left( x,0\right)  &=&u\left( x\right) .  \notag
\end{eqnarray}%
Note that the differential equation above can be written as%
\begin{equation}
\mathbf{\mathcal{L}}_{w,\sigma }U:=-\frac{1}{\mathbf{w}\left( x,y\right) }%
\func{div}\mathbf{w}\left( x,y\right) \mathcal{A}U\left( x,y\right) =0
\label{eq-div}
\end{equation}%
where, for all $x\in \Omega $ and $y\in \mathbb{R}$ 
\begin{equation*}
\mathcal{A}\left( x,y\right) =\left( 
\begin{array}{cc}
\mathbf{A}\left( x\right)  & 0 \\ 
0 & 1%
\end{array}%
\right) \qquad \text{and}\qquad \mathbf{w}\left( x,y\right) =w\left(
x\right) \left\vert y\right\vert ^{1-2\sigma }.
\end{equation*}%
This implies that $U\left( x,y\right) $ is a local solution of $\mathbf{%
\mathcal{L}}_{w,\sigma }U=0$ in $\Omega \times \left( 0,\infty \right) $,
i.e. for all ${\varphi \in }\mathcal{H}_{0}^{1}\left( \mathbf{w},\Omega
\times \left( 0,\infty \right) \right) $ we have that%
\begin{equation}
\int_{\Omega \times \left( 0,\infty \right) }\mathcal{A}\left( x,y\right)
\nabla U\left( x,y\right) \cdot \nabla {\varphi }\left( x,y\right) ~d\mathbf{%
w}=0.  \label{eq-U}
\end{equation}%
We now extend $U\left( x,y\right) $ for negative values of $y$ in an even
way an call this extension $\widetilde{U}$, i.e.%
\begin{equation*}
\widetilde{U}\left( x,y\right) =\left\{ 
\begin{array}{lr}
U\left( x,y\right)  & \qquad y\geq 0 \\ 
U\left( x,-y\right)  & \qquad y<0%
\end{array}%
\right. .
\end{equation*}%
We claim that $\widetilde{U}\left( x,y\right) $ is a local weak solution of $%
\mathbf{\mathcal{L}}_{w,\sigma }\widetilde{U}=0$ in $\Omega \times \mathbb{R}
$. Indeed, this follows the same way as in \cite{caffarelli-silvestre-07}.
Let ${\varphi }\left( x,y\right) \in \mathcal{C}_{0}^{\infty }\left( \Omega
\times \mathbb{R}\right) $ (note that $\mathcal{C}_{0}^{\infty }\left(
\Omega \times \mathbb{R}\right) $ is dense in $\mathcal{H}_{0}^{1}\left( 
\mathbf{w},\Omega \times \mathbb{R}\right) $) and let $\eta \left( y\right)
=\eta \left( \left\vert y\right\vert \right) $ be a smooth even cutoff
function such that $0\leq \eta \leq 1$, $\eta \equiv 1$ for $\left\vert
y\right\vert \leq 1$, $\eta \equiv 0$ for $\left\vert y\right\vert \geq 2$
and $\left\vert \eta ^{\prime }\right\vert \leq 2$. Set also $\eta
_{\varepsilon }\left( y\right) =\eta \left( \frac{y}{\varepsilon }\right) $
for all $\varepsilon >0$. Then by (\ref{eq-U}) and (\ref{eq-ord}) it follows
that%
\begin{eqnarray*}
&&\int_{\Omega \times \mathbb{R}}\mathcal{A}\left( x,y\right) \nabla 
\widetilde{U}\left( x,y\right) \cdot \nabla {\varphi }\left( x,y\right) ~d%
\mathbf{w} \\
&=&\int_{\Omega \times \mathbb{R}}\mathcal{A}\left( x,y\right) \nabla 
\widetilde{U}\left( x,y\right) \cdot \nabla \eta _{\varepsilon }\left(
y\right) {\varphi }\left( x,y\right) ~d\mathbf{w} \\
&=&\int_{\varepsilon \leq \left\vert y\right\vert \leq 2\varepsilon
}\left\vert y\right\vert ^{1-2\sigma }\eta _{\varepsilon }\left( y\right)
\int_{\Omega }\mathbf{A}\left( x\right) \nabla _{x}\widetilde{U}\left(
x,y\right) \cdot \nabla _{x}{\varphi }\left( x,y\right) ~dw~dy \\
&&+\int_{\varepsilon \leq \left\vert y\right\vert \leq 2\varepsilon
}\int_{\Omega }\left\vert y\right\vert ^{1-2\sigma }\widetilde{U_{y}}\left(
x,y\right) \cdot \frac{\partial }{\partial y}\left( \eta _{\varepsilon
}\left( y\right) {\varphi }\left( x,y\right) \right) ~dw~dy \\
&=&\int_{\varepsilon \leq \left\vert y\right\vert \leq 2\varepsilon
}\left\vert y\right\vert ^{1-2\sigma }\eta _{\varepsilon }\left( y\right)
\int_{\Omega }\left( \frac{1-2\sigma }{y}\widetilde{U_{y}}\left( x,y\right)
+U_{yy}\left( x,y\right) \right) {\varphi }\left( x,y\right) ~dw~dy \\
&&+\int_{\varepsilon \leq \left\vert y\right\vert \leq 2\varepsilon
}\int_{\Omega }\left\vert y\right\vert ^{1-2\sigma }\widetilde{U_{y}}\left(
x,y\right) \cdot \frac{\partial }{\partial y}\left( \eta _{\varepsilon
}\left( y\right) {\varphi }\left( x,y\right) \right) ~dw~dy.
\end{eqnarray*}%
We split the last two integrals into two, depending on the sign of the
integrand $y$. When $y>0$, by integration by parts have%
\begin{eqnarray*}
&&\int_{\varepsilon }^{2\varepsilon }y^{1-2\sigma }\eta _{\varepsilon
}\left( y\right) \int_{\Omega }\left( \frac{1-2\sigma }{y}\widetilde{U_{y}}%
\left( x,y\right) +\frac{\partial ^{2}}{\partial y^{2}}\widetilde{U}\left(
x,y\right) \right) {\varphi }\left( x,y\right) ~dw~dy \\
&&+\int_{\varepsilon }^{2\varepsilon }\int_{\Omega }y^{1-2\sigma }\widetilde{%
U_{y}}\left( x,y\right) \cdot \frac{\partial }{\partial y}\left( \eta
_{\varepsilon }\left( y\right) {\varphi }\left( x,y\right) \right) ~dw~dy \\
&=&\int_{\Omega }\int_{\varepsilon }^{2\varepsilon }\eta _{\varepsilon
}\left( y\right) \frac{\partial }{\partial y}\left( y^{1-2\sigma
}U_{y}\left( x,y\right) \right) {\varphi }\left( x,y\right) ~dy~dw \\
&&+\int_{\varepsilon }^{2\varepsilon }\int_{\Omega }y^{1-2\sigma
}U_{y}\left( x,y\right) \cdot \frac{\partial }{\partial y}\left( \eta
_{\varepsilon }\left( y\right) {\varphi }\left( x,y\right) \right) ~dw~dy \\
&=&\int_{\Omega }\left. \eta _{\varepsilon }\left( y\right) \left(
y^{1-2\sigma }U_{y}\left( x,y\right) \right) {\varphi }\left( x,y\right)
\right\vert _{y=\varepsilon }^{y=2\varepsilon }~dw \\
&=&\int_{\Omega }\varepsilon ^{1-2\sigma }U_{y}\left( x,\varepsilon \right) {%
\varphi }\left( x,\varepsilon \right) ~dw.
\end{eqnarray*}%
A similar treatment for the region where $y<0$ yields%
\begin{eqnarray*}
&&\int_{\Omega \times \mathbb{R}}\mathcal{A}\left( x,y\right) \nabla 
\widetilde{U}\left( x,y\right) \cdot \nabla {\varphi }\left( x,y\right) ~d%
\mathbf{w} \\
&=&\int_{\Omega }\varepsilon ^{1-2\sigma }U_{y}\left( x,\varepsilon \right)
\left( {\varphi }\left( x,\varepsilon \right) +{\varphi }\left(
x,-\varepsilon \right) \right) ~dw.
\end{eqnarray*}%
Since $\varepsilon >0$ is arbitrary and $\varepsilon ^{1-2\sigma }\frac{%
\partial U}{\partial y}\left( x,\varepsilon \right) \rightarrow \left(
L_{w}\right) ^{\sigma }u=0$ by (\ref{limits}), we have that%
\begin{equation*}
\int_{\Omega \times \mathbb{R}}\mathcal{A}\left( x,y\right) \nabla 
\widetilde{U}\left( x,y\right) \cdot \nabla {\varphi }\left( x,y\right) ~d%
\mathbf{w}=0
\end{equation*}%
which show that $\widetilde{U}\left( x,y\right) $ is a local weak solution
of $\mathbf{\mathcal{L}}_{w,\sigma }\widetilde{U}=0$ in $\Omega \times 
\mathbb{R}$. We note that $w\in A_{2}\left( \mathbb{R}^{n}\right)
\Longrightarrow \mathbf{w}=w\left( x\right) \left\vert y\right\vert
^{1-2\sigma }\in A_{2}\left( \mathbb{R}^{n+1}\right) $ for all $0<\sigma <1$%
. Indeed, for fixed $\left( x_{0},y_{0}\right) \in \mathbb{R}^{n+1}$ and $r>0
$%
\begin{eqnarray*}
&&\left( \frac{1}{\left\vert B_{r}\left( x_{0},y_{0}\right) \right\vert }%
\int\limits_{B_{r}\left( x_{0},y_{0}\right) }\mathbf{w}\left( x,y\right)
~dxdy\right) ^{\frac{1}{2}}\left( \frac{1}{\left\vert B_{r}\left(
x_{0},y_{0}\right) \right\vert }\int\limits_{B_{r}\left( x_{0},y_{0}\right) }%
\frac{1}{\mathbf{w}\left( x,y\right) }~dxdy\right) ^{\frac{1}{2}} \\
&\leq &C_{n}\left( \left( \frac{1}{\left\vert B_{r}\left( x_{0}\right)
\right\vert }\int_{\left\vert x-x_{0}\right\vert <r}w\left( x\right)
~dx\right) \left( \frac{1}{\left\vert B_{r}\left( x_{0}\right) \right\vert }%
\int_{\left\vert x-x_{0}\right\vert <r}\frac{1}{w\left( x\right) }~dx\right)
\right) ^{\frac{1}{2}} \\
&&\times \left( \left( \frac{1}{2r}\int_{\left\vert y-y_{0}\right\vert
<r}\left\vert y\right\vert ^{1-2\sigma }dy\right) \left( \frac{1}{2r}%
\int_{\left\vert y-y_{0}\right\vert <r}\left\vert y\right\vert ^{2\sigma
-1}dy\right) \right) ^{\frac{1}{2}} \\
&\leq &C_{n}\left[ w\right] _{A_{2}\left( \mathbb{R}^{n}\right) }\left[
\left\vert y\right\vert ^{2\sigma -1}\right] _{A_{2}\left( \mathbb{R}\right)
}<\infty .
\end{eqnarray*}%
By Theorem \ref{theorem-FKS} it follows that for all $x_{0}\in \Omega $ and $%
r>0$ such that $B_{2r}\left( x_{0}\right) \subset \Omega $%
\begin{eqnarray}
\max_{B_{r}\left( x_{0}\right) }\left\vert u\left( x\right) \right\vert 
&\leq &\max_{B_{r}\left( x_{0},0\right) }\left\vert \widetilde{U}\left(
x,y\right) \right\vert   \label{eq-this} \\
&\leq &M\frac{1}{\mathbf{w}\left( B_{2r}\left( x_{0},0\right) \right) }%
\int_{B_{2r}\left( x_{0},0\right) }\left\vert \widetilde{U}\left( x,y\right)
\right\vert ^{2}d\mathbf{w}\left( x,y\right) .  \notag
\end{eqnarray}%
Since any $A_{2}$ weight is doubling, there exists a constant $D_{w}>1$ such
that%
\begin{eqnarray*}
\mathbf{w}\left( B_{2r}\left( x_{0},0\right) \right)  &\geq &D_{w}^{-1}%
\mathbf{w}\left( B_{2r}\left( x_{0}\right) \times \left[ -2r,2r\right]
\right)  \\
&=&D_{w}^{-1}\int_{-2r}^{2r}\left\vert y\right\vert ^{1-2\sigma
}\int_{B_{2r}\left( x_{0}\right) }w\left( x\right) dx~dy \\
&=&D_{w}^{-1}\frac{\left( 2r\right) ^{2-2\sigma }}{1-\sigma }w\left(
B_{2r}\left( x_{0}\right) \right) .
\end{eqnarray*}%
We also have that%
\begin{eqnarray}
\int_{B_{2r}\left( x_{0},0\right) }\left\vert \widetilde{U}\left( x,y\right)
\right\vert ^{2}d\mathbf{w}\left( x,y\right)  &\leq
&\int_{-2r}^{2r}\left\vert y\right\vert ^{1-2\sigma }\int_{B_{2r}\left(
x_{0}\right) }\left\vert \widetilde{U}\left( x,y\right) \right\vert
^{2}w\left( x\right) dx~dy  \notag \\
&=&2\int_{0}^{2r}y^{1-2\sigma }\int_{B_{2r}\left( x_{0}\right) }\left\vert
U\left( x,y\right) \right\vert ^{2}w\left( x\right) dx~dy  \notag \\
&\leq &2\int_{0}^{2r}y^{1-2\sigma }\left\Vert U\left( \cdot ,y\right)
\right\Vert _{L^{2}\left( w,\Omega \right) }^{2}~dy  \notag \\
&\leq &C\frac{\left( 2r\right) ^{2-2\sigma }}{1-\sigma }\left\Vert
u\right\Vert _{L^{2}\left( w,\Omega \right) }^{2}.  \label{eq-that}
\end{eqnarray}%
Putting these inequalities together with (\ref{eq-this}) yields%
\begin{equation*}
\max_{B_{r}\left( x_{0}\right) }\left\vert u\left( x\right) \right\vert \leq 
\frac{CMD_{w}}{w\left( B_{2r}\left( x_{0}\right) \right) }\left\Vert
u\right\Vert _{L^{2}\left( w,\Omega \right) }^{2}.
\end{equation*}%
Which shows that $u$ is locally bounded in $\Omega $.

On the other hand we note that since the semigroup $e^{-tL_{w}}$ is positive
(cf. Section 4.5 in \cite{ouhabaz05}), hence if $u$ is moreover nonnegative
it follows that $\widetilde{U}$ is nonnegative. Then, by Theorem \ref%
{theorem-FKS}%
\begin{equation*}
\max_{B_{r}\left( x_{0}\right) }u\left( x\right) \leq \max_{B_{r}\left(
x_{0},0\right) }\widetilde{U}\left( x,y\right) \leq M\min_{B_{r}\left(
x_{0},0\right) }\widetilde{U}\left( x,y\right) \leq M\min_{B_{r}\left(
x_{0}\right) }u\left( x\right) .
\end{equation*}%
This proves the Harnack's estimate for nonnegative solutions of $\left(
L_{w}\right) ^{\sigma }u=0$. H\"{o}lder's continuity of solutions follows
directly from these scale invariant estimates.
\end{proof}

The proof of the boundary Harnack principle for fractional powers of $L_{w}$
(Theorem \ref{theorem-BH}) is just the same as the proof of Theorem 5.3 in 
\cite{caffarelli-silvestre-07}, so we only provide a sketch.

\begin{proof}[Proof of Theorem \protect\ref{theorem-BH}]
Let $\widetilde{U}\left( x,y\right) $ be as in the previous proof, then $%
\tilde{U}$ is a solution of (\ref{eq-div}) in $\Omega \times \mathbb{R}$.
Applying Theorem 2.4.6 from \cite{fabes-kenig-serapioni82} (boundary H\"{o}%
lder continuity) in the set $G_{1}$, where%
\begin{equation*}
G_{\rho }=B_{\rho }\left( x_{0}\right) \times \left( -\rho ,\rho \right)
\bigcap \Omega \times \mathbb{R},
\end{equation*}%
yields%
\begin{equation}
\sup_{G_{\rho }}\widetilde{U}\left( x,y\right) -\inf_{G_{\rho }}\widetilde{U}%
\left( x,y\right) \leq M\left( \int_{G_{\frac{1}{2}}}\left\vert \widetilde{U}%
\left( x,y\right) \right\vert ^{2}d\mathbf{w}\left( x,y\right) \right) ^{%
\frac{1}{2}}\rho ^{\alpha }.  \label{eq-HU}
\end{equation}%
Since   
\begin{equation*}
\sup_{B_{\rho }\left( x_{0}\right) \bigcap \Omega }u\left( x\right)
-\inf_{B_{\rho }\left( x_{0}\right) \bigcap \Omega }u\left( x\right) \leq
\sup_{G_{\rho }}\widetilde{U}\left( x,y\right) -\inf_{G_{\rho }}\widetilde{U}%
\left( x,y\right) 
\end{equation*}%
the Theorem follows by bounding the right hand side in (\ref{eq-HU}) in a
similar way as we obtained (\ref{eq-that}).%
\begin{eqnarray}
\int_{G_{\frac{1}{2}}}\left\vert \widetilde{U}\left( x,y\right) \right\vert
^{2}d\mathbf{w}\left( x,y\right)  &\leq &\int_{-\frac{1}{2}}^{\frac{1}{2}%
}\left\vert y\right\vert ^{1-2\sigma }\int_{B_{\frac{1}{2}}\left(
x_{0}\right) \bigcap \Omega }\left\vert \widetilde{U}\left( x,y\right)
\right\vert ^{2}w\left( x\right) dx~dy  \notag \\
&=&2\int_{0}^{\frac{1}{2}}y^{1-2\sigma }\int_{B_{\frac{1}{2}}\left(
x_{0}\right) \bigcap \Omega }\left\vert U\left( x,y\right) \right\vert
^{2}w\left( x\right) dx~dy  \notag \\
&\leq &2\int_{0}^{2r}y^{1-2\sigma }\left\Vert U\left( \cdot ,y\right)
\right\Vert _{L^{2}\left( w,\Omega \right) }^{2}~dy  \notag \\
&\leq &C\frac{\left( 2r\right) ^{2-2\sigma }}{1-\sigma }\left\Vert
u\right\Vert _{L^{2}\left( w,\Omega \right) }^{2}.
\end{eqnarray}
\end{proof}

\bigskip 

\bigskip 

Let $0<\sigma \leq 1$, for every $\mathbf{A}\in \mathcal{F}_{n}\left(
\lambda ,\Lambda \right) $ and $w\in A_{2}$. Suppose $u\in \mathcal{D}\left(
L^{\sigma }\right) $ is a function on $\mathbb{R}^{n}$ such that $\left(
L_{w}\right) ^{\sigma }u=0$ in a domain $\Omega $, and suppose that for some 
$x_{0}\in \Omega $, $u=0$ on $B_{1}\left( x_{0}\right) \backslash \Omega $
where $\partial \Omega \bigcap B_{1}\left( x_{0}\right) $ is given by a
Lipschitz graph with constant less than $1$. Then there exist constants $M>0$
and $0<\alpha <1$ depending on $\lambda ,\Lambda ,\left[ w\right] _{A_{2}}$,
and $\sigma $ such that for all $0<\rho <\frac{1}{2}$%
\begin{equation*}
\sup_{\Omega \bigcap B_{\rho }\left( x_{0}\right) }u-\inf_{\Omega \bigcap
B_{\rho }\left( x_{0}\right) }u\leq M\left( \frac{1}{w\left( B_{\frac{1}{2}%
}\left( x_{0}\right) \right) }\int_{B_{\frac{1}{2}}\left( x_{0}\right)
}u^{2}dw\right) ^{\frac{1}{2}}\rho ^{\alpha }.
\end{equation*}

\subsection{Non isotropic operators\label{section-noni}}

In \cite{sawyer-wheeden-06} Sawyer and Wheeden consider the general linear
second order equations%
\begin{equation}
\mathcal{L}u=-\func{div}\mathbf{B}\left( x\right) \nabla
u+\sum_{i=1}^{N}b_{i}R_{i}u+\sum_{i=1}^{N}S_{i}^{\prime
}c_{i}u+du=f+\sum_{i=1}^{N}T_{i}^{\prime }g_{i},  \label{eq-subelliptic}
\end{equation}%
for which the principal part is nonnegative but not necessarily strongly
elliptic. More precisely, these authors assumed the following conditions:%
\renewcommand{\theenumi}{\Alph{enumi}}%

\begin{enumerate}
\item \label{noni-A}$\mathbf{B}$ is a bounded measurable nonnegative
semidefinite matrix,

\item \label{noni-B}$\left\{ R_{i}\right\} _{i=1}^{N}$, $\left\{
S_{i}\right\} _{i=1}^{N}$, and $\left\{ T_{i}\right\} _{i=1}^{N}$, are
collections of vector fields subunit with respect to $\mathbf{B}\left(
x\right) $, i.e. 
\begin{equation*}
\left( X\left( x\right) \cdot \xi \right) ^{2}\leq \xi ^{\prime }\mathbf{B}%
\left( x\right) \xi \qquad \text{for all }\xi \in \mathbb{R}^{n},
\end{equation*}%
$X=R_{i},S_{i},T_{i}$, $i=1,\cdots ,N$, and all $x$ in a domain (open and
connected set) $\Omega \subset \mathbb{R}^{n}$;

\item \label{noni-C}the operator coefficients $\mathbf{b}=\left\{
b_{i}\right\} _{i=1}^{N}$, $\mathbf{c}=\left\{ c_{i}\right\} _{i=1}^{N}$ , $%
d $, and the inhomogeneous data $\mathbf{g}=\left\{ g_{i}\right\} _{i=1}^{N}$
and $f$ are measurable.

\item \label{noni-D}The coefficients and data moreover satisfy%
\begin{eqnarray*}
\left\Vert d\right\Vert _{L^{\frac{q}{2}}\left( \Omega \right) }+\left\Vert 
\mathbf{b}\right\Vert _{L^{q}\left( \Omega \right) }+\left\Vert \mathbf{c}%
\right\Vert _{L^{q}\left( \Omega \right) } &:&\equiv N_{q}<\infty \\
\left\Vert f\right\Vert _{L^{\frac{q}{2}}\left( \Omega \right) }+\left\Vert 
\mathbf{g}\right\Vert _{L^{q}\left( \Omega \right) } &:&\equiv N_{q}^{\prime
}<\infty
\end{eqnarray*}%
for some $q\geq 2$.
\end{enumerate}

When ellipticity is allowed to degenerate, the concept of weak solution must
be adapted to the geometry induced by the principal part of the operator,
namely, the geometry of the subunit metric with respect to the matrix $%
\mathbf{B}$, as described in the introduction. The natural space for
solutions is also determined by the quadratic form given by $\mathbf{B}$, we
let $W_{\mathbf{B}}^{1,2}\left( \Omega \right) $ be the space of square
integrable measurable functions $f$ such that their gradient belongs to the
space $\mathbf{L}_{\mathbf{B}}^{2}$ given by measurable vector functions $%
\mathbf{v}\left( x\right) $ such that%
\begin{equation*}
\left\Vert \mathbf{v}\right\Vert _{\mathbf{L}_{\mathbf{B}}^{2}}^{2}=\int_{%
\Omega }\mathbf{v}\cdot \mathbf{Bv}~dx<\infty .
\end{equation*}%
We say that $u$ is a solution of (\ref{eq-subelliptic}) in $\Omega $ if $%
u\in W_{\mathbf{B}}^{1,2}\left( \Omega \right) $ and 
\begin{eqnarray*}
&&\int \nabla u~\mathbf{B}\nabla w+\int \sum_{i=1}^{N}b_{i}\left(
R_{i}u\right) w+\int \sum_{i=1}^{N}c_{i}u\left( S_{i}w\right) +\int duw \\
&=&\int fw+\int \sum_{i=1}^{N}g_{i}\left( T_{i}w\right) ,
\end{eqnarray*}%
for all nonnegative $w\in \left( W_{\mathbf{B}}^{1,2}\right) _{0}\left(
\Omega \right) $. We will further assume that the matrix $\mathbf{B}$ is
equivalent to a special kind of diagonal matrices (see condition (\ref%
{condition-D}) below) what will ensure that $W_{\mathbf{B}}^{1,2}\left(
\Omega \right) $ is a Hilbert space with inner product given by%
\begin{equation}
\left\langle f,g\right\rangle _{W_{\mathbf{B}}^{1,2}\left( \Omega \right)
}=\int_{\Omega }fg~dx+\int_{\Omega }\nabla f\cdot \mathbf{B}\nabla g~dx.
\label{inner}
\end{equation}%
See \cite{sawyer-wheeden-10, MRW15} for details.

Note that the operator $L$ is given by the bilinear form 
\begin{equation}
\mathcal{E}\left( u,w\right) =\int \nabla u~\mathbf{B}\nabla w+\int
\sum_{i=1}^{N}b_{i}\left( R_{i}u\right) w+\int \sum_{i=1}^{N}c_{i}u\left(
S_{i}w\right) +\int duw  \label{Eni}
\end{equation}%
on $\mathcal{F=}\left( W_{\mathbf{B}}^{1,2}\right) \left( \Omega \right)
\subset L^{2}\left( \Omega \right) $. We will assume this bilinear form is
nonnegative; for this would suffice that $\mathbf{b}$, $\mathbf{c}$, and $d$
vanish, or that $d\geq d_{0}>0$ for some constant $d_{0}$, and $\left\Vert 
\mathbf{b}\right\Vert _{\infty },\left\Vert \mathbf{c}\right\Vert _{\infty }$
small enough in view of the subunit condition (2) above. Indeed, for the
second term in the (\ref{Eni}) we have 
\begin{eqnarray*}
\left\vert \int \sum_{i=1}^{N}b_{i}\left( R_{i}u\right) u\right\vert &\leq
&\sum_{i=1}^{N}\left( \int \left\vert R_{i}u\right\vert ^{2}\right) ^{\frac{1%
}{2}}\left( \int b_{i}^{2}u^{2}\right) ^{\frac{1}{2}} \\
&\leq &N\left\Vert \mathbf{b}\right\Vert _{\infty }\left( \int \nabla u~%
\mathbf{B}\nabla u\right) ^{\frac{1}{2}}\left( \int u^{2}\right) ^{\frac{1}{2%
}},
\end{eqnarray*}%
with a similar estimate for the third term in (\ref{Eni}); thus, 
\begin{eqnarray}
\mathcal{E}\left( u,u\right) &\geq &\int \nabla u~\mathbf{B}\nabla w
\label{Euu} \\
&&+N\left( \left\Vert \mathbf{b}\right\Vert _{\infty }+\left\Vert \mathbf{c}%
\right\Vert _{\infty }\right) \left( \int \nabla u~\mathbf{B}\nabla u\right)
^{\frac{1}{2}}\left( \int u^{2}\right) ^{\frac{1}{2}}+\int du^{2}.  \notag
\end{eqnarray}%
It easily follows that the right hand side is nonnegative under the above
assumptions. Hence, $\mathcal{E}$ satisfies the lower bound (\ref{E1-lb}) in
Definition \ref{def-Dirichlet}, while the sector condition (\ref{E2-sector})
can similarly be verified. We will assume further structural assumptions to
guarantee a-priori Harnack estimated for solutions. First we must include
some more definitions and background.

\begin{definition}[Reverse H\"{o}lder infinity]
A nonnegative function $a\left( t\right) $ defined on an open subset $J$ of $%
\mathbb{R}$ satisfies the reverse H\"{o}lder condition of infinite order if%
\begin{equation*}
\limfunc{ess}\sup_{I}a\left( t\right) \leq C\frac{1}{\left\vert I\right\vert 
}\int_{I}a\left( t\right) ~dt
\end{equation*}%
for all intervals $I\subset J$. In such case, we say that $a\in RH^{\infty
}\left( J\right) $.
\end{definition}

\begin{remark}
All positive powers are in $RH^{\infty }$.
\end{remark}

\begin{definition}[Flag condition - \protect\cite{sawyer-wheeden-06}
Definition 12]
A collection of continuous vector fields satisfies the flag condition at $%
x\in \Omega $ if for each index set $\emptyset \subset \mathcal{I}%
\varsubsetneq \left\{ 1,2,\dots ,n\right\} $, there is $j\notin \mathcal{I}$
such that for any neighbourhood $\mathcal{N}$ of $x$ in $\Omega $, $a_{j}$
does not vanish identically on $\left( x+\mathcal{V}_{\mathcal{I}}\right)
\bigcap \mathcal{N}$ where $\mathcal{V}_{\emptyset }=\left\{ 0\right\} $ and 
$\mathcal{V}_{\mathcal{I}}=\limfunc{span}\left\{ \mathbf{e}_{i}:i\in 
\mathcal{I}\right\} $, $\mathbf{e}_{i}=\left( 0,\dots ,0,1,0,\dots ,0\right) 
$ (with $1$ in the $i^{th}$ position). The vector fields $X_{i}$ satisfy the
flag condition in $\Omega $ if they satisfy the flag condition at every
point $x\in \Omega $.
\end{definition}

The flag condition ensures that the flow of the vector fields $X_{j}$ does
not get "trapped" into any variety of dimension less than $n$. This
condition is necessary for subellipticity of operators given by diagonal,
Lipschitz, $RH^{\infty }$, vector fields (see \cite{sawyer-wheeden-06},
Theorem 16).

We assume that the principal part $-\func{div}\mathbf{B}\nabla $ of our
operator $\mathcal{L}$ satisfies the following structural condition:

\begin{condition}
\label{condition-D}There exist nonnegative Lipschitz functions $a_{j}$ on $%
\Omega \subset \mathbb{R}^{n}$, $j=1,\cdots ,n$, such that 
\renewcommand{\theenumi}{\Roman{enumi}}%

\begin{enumerate}
\item \label{condition-D-A}for every compact set $K\subset \Omega $, $%
a_{j}\in RH^{\infty }\left( K\right) $ in each variable $x_{i}$ with $i\neq
j $, uniformly in the remaining variables;

\item \label{condition-D-B}the set $\mathbb{X}=\left\{ X_{j}\right\}
_{j=1}^{n}$ of vector fields $X_{j}=a_{j}\frac{\partial }{\partial x_{j}}$
satisfies the flag condition in $\Omega $;

\item \label{condition-D-C}there exist constants $0<c_{\mathbf{B}}\leq C_{%
\mathbf{B}}<\infty $ such that the matrix $\mathbf{B}\left( x\right) $
satisfies the upper and lower bounds in $\Omega $ 
\begin{equation}
c_{\mathbf{B}}\sum_{j=1}^{n}a_{j}\left( x\right) \xi _{j}^{2}\leq \xi
^{\prime }\mathbf{B}\left( x\right) \xi \leq C_{\mathbf{B}%
}\sum_{j=1}^{n}a_{j}\left( x\right) \xi _{j}^{2}.  \label{quadratic-equiv-2}
\end{equation}
\end{enumerate}
\end{condition}

In particular, this structure allows for different order of vanishing of the
eigenvalues $a_{j}$ of $\mathbf{B}$, what was not permitted to the weighted
elliptic operators in \cite{fabes-kenig-serapioni82} treated in Section \ref%
{section-weighted-elliptic}.

The following is an example of a diagonal system of vector fields which
satisfies the flag condition and condition \ref{condition-D}.

\begin{example}
\label{example-diag}Suppose $a_{j}\left( x\right) $, $j=1,\cdots ,n$, are
nonnegative Lispchitz functions, that for each $x_{0}\in \Omega $ there
exists a permutation $\tau =\tau _{x_{0}}$ of the set $\left\{ 1,\cdots
,n\right\} $ and a neighbourhood $\mathcal{N}_{x_{0}}\subset \Omega $ of $%
x_{0}$ such that in $\mathcal{N}_{x_{0}}$ we have for $\left( y_{1},\dots
,y_{n}\right) =\left( x_{\tau \left( 1\right) },\dots ,x_{\tau \left(
n\right) }\right) $ and $\tilde{a}_{j}\left( y\right) =a_{\tau \left(
j\right) }\left( \tau ^{-1}\left( y\right) \right) $:

\begin{itemize}
\item $\tilde{a}_{1}\left( y\right) \approx 1$, and $\tilde{a}_{j}\left(
y\right) =\tilde{a}_{j}\left( y_{1},\dots ,y_{j-1}\right) $ for $j=2,\cdots
,n$;

\item $\tilde{a}_{j}$ has isolated zeroes in their variables; i.e. 
\begin{equation*}
Z_{j}=\left\{ \left( y_{1},\dots ,y_{j-1}\right) :\tilde{a}_{j}\left(
y_{1},\dots ,y_{j-1}\right) =0\right\} \bigcap \mathcal{N}
\end{equation*}%
is a discrete set in $\mathbb{R}^{j-1}$ for $j=2,\cdots ,n$;

\item $\tilde{a}_{j}$ is locally homogeneous of finite type: if $z\in $ $%
Z_{j}$ then 
\begin{equation*}
\tilde{a}_{j}\left( w-z\right) \approx \left\vert w-z\right\vert
^{k_{j}}\qquad \text{ }w\text{ near }z
\end{equation*}%
for some integers $k_{j}\geq 1$, $2\leq j\leq n$. \newline
Then $X_{j}=a_{j}\frac{\partial }{\partial x_{j}}$, $1\leq j\leq n$, is a
collection of vector fields which satisfies the flag condition and (\ref%
{condition-D-A}) from condition \ref{condition-D}.
\end{itemize}
\end{example}

\begin{proof}
As noted in \cite{sawyer-wheeden-06} (Remark 13), to check that $\left\{
X_{j}\right\} _{j=1}^{n}$ satisfies the flag condition at a point $x_{0}$ it
suffices show that there exist an increasing sequence of index sets 
\begin{equation*}
\emptyset \neq \mathcal{I}_{1}\subsetneqq \mathcal{I}_{2}\subsetneqq \dots
\subsetneqq \mathcal{I}_{n}=\left\{ 1,\cdots ,n\right\} ,
\end{equation*}%
such that for $\mathcal{V}_{0}=\left\{ 0\right\} $ and $\mathcal{V}_{j}=%
\limfunc{span}\left\{ e_{i}:i\in \mathcal{I}_{j}\right\} $, $a_{i}$ does not
vanish identically on $\left( x_{0}+\mathcal{V}_{j}\right) \bigcap \mathcal{N%
}$ for any neighbourhood $\mathcal{N}$ of $x_{0}$ and any $i\in \mathcal{I}%
_{j+1}$, $j=0,1,\cdots ,n-1$. It suffices to check this when the permutation 
$\tau $ is the identity and the point $x_{0}$ is the origin. From the
definition of $a_{j}$, taking 
\begin{equation*}
\mathcal{I}_{j}=\left\{ 1,\dots ,j\right\} ,\qquad j=1,\cdots ,n,
\end{equation*}%
we have that if $i+1\in \mathcal{I}_{j}$, for some $j=2,\cdots ,n$, then $%
1\leq i\leq j-1$. Let $0_{i}$ be the origin in $\mathbb{R}^{i}$. If $%
a_{i}\left( 0_{i}\right) >0$ then there is nothing to prove since $a_{i}$ is
continuous. If $a_{i}\left( 0_{i}\right) =0$ then $a_{i}\left( w\right)
\approx \left\vert w\right\vert ^{k_{i}}$ for $w\in \mathbb{R}^{i}$ near $%
0_{i}$ and since $\mathcal{V}_{j}\bigcap \mathcal{N}\supset \mathbb{R}%
^{i}\bigcap \mathcal{N}$ we see that $a_{i}$ does not vanish identically on $%
\left( x_{0}+\mathcal{V}_{j}\right) \bigcap \mathcal{N}$ for any
neighbourhood $\mathcal{N}$.

Finally, if $a_{i}\left( 0_{i}\right) >0$ then $a_{i}$ is locally constant
and therefore $a_{i}$ satisfies (\ref{condition-D-A}) from condition \ref%
{condition-D}. On the other hand, if $a_{i}\left( 0_{i}\right) =0$ and $%
i<\ell \leq n$ then $a_{i}$ is constant in the variable $w_{\ell }$ so $%
a_{i} $ is in $RH^{\infty }$ of this variable independently of the remaining
variables, while if $a_{i}\left( 0_{i}\right) =0$ and $1\leq \ell \leq i$
then 
\begin{equation*}
a_{i}\left( w\right) \approx \left( w_{1}^{2}+\dots +w_{\ell }^{2}+\dots
+w_{i}^{2}\right) ^{\frac{k_{i}}{2}}
\end{equation*}%
so $a_{i}$ is in $RH^{\infty }$ of the variable $w_{\ell }$, uniformly on
the remaining variables.
\end{proof}

The following Harnack inequality can be found in \cite{sawyer-wheeden-06}
(Propositions 58 and 67 and Theorems 61 and 82), see also \cite{MRW15}.

\begin{theorem}[Harnack's inequality for subelliptic almost-diagonal
operators]
\label{theorem-SW}Let $\mathcal{L}$ be given by (\ref{eq-subelliptic}) where
the coefficients satisfy (\ref{noni-A}), (\ref{noni-B}), (\ref{noni-C}), and
(\ref{noni-D}), and $\mathbf{B}\left( x\right) $ satisfies condition \ref%
{condition-D} in a domain $\Omega \subset \mathbb{R}^{n}$. Then there exist
a constant $C_{H}>0$ such that every weak solution of $\mathcal{L}u=0$ in $%
B_{2r}\left( y\right) \subset \Omega $, satisfies%
\begin{equation*}
\limfunc{ess}\sup_{B_{r}}u\leq C_{H}\left( \frac{1}{\left\vert B_{2r}\left(
y\right) \right\vert }\int_{B_{2r}\left( y\right) }\left\vert u\right\vert
^{2}dx\right) ^{\frac{1}{2}}.
\end{equation*}%
where the balls $B_{r}$ are the subunit metric balls of the metric induced
by $\mathbb{X}$. Moreover, if $u$ is nonnegative, then we also have that%
\begin{equation*}
\limfunc{ess}\sup_{B_{r}}u\leq C_{H}~\limfunc{ess}\inf_{B_{r}}u.
\end{equation*}
\end{theorem}

Under Condition \ref{condition-D}, the space $W_{\mathbf{B}}^{1,2}\left(
\Omega \right) $ is a Hilbert space with inner product given by the left
hand side of (\ref{inner}) (see \cite{sawyer-wheeden-10}, Theorem 2 and
Section 3); from this it follows that the completeness condition (\ref%
{E3-complete}) holds for the form $\mathcal{E}$ in (\ref{Euu}). Thus, $%
\mathcal{E}$ satisfies the hypotheses of Corollary \ref{coro-sectorial}, and
therefore $L$ is a sectorial operator on $W_{\mathbf{B}}^{1,2}\left( \Omega
\right) $. This allows us to apply the functional calculus of Section \ref%
{section-calculus} to $L$.

With these preliminaries laid down, we can not prove Theorem \ref%
{theorem-noniso}; we will obtain this theorem as a consequence of the
following more general result:

\begin{theorem}
\label{theorem-noniso-2}Let $\mathcal{L}$ be given by (\ref{eq-subelliptic})
where the coefficients satisfy (\ref{noni-A}), (\ref{noni-B}), (\ref{noni-C}%
), and (\ref{noni-D}), the bilinear form $\mathcal{E}$ given by (\ref{Eni})
is nonnegative, and $\mathbf{B}\left( x\right) $ satisfies condition \ref%
{condition-D} in a domain $\Omega \subset \mathbb{R}^{n}$. Suppose $u\in 
\mathcal{D}\left( L^{\frac{1}{2}}\right) $ and that $L^{\frac{1}{2}}u=0$ in
an open set $\Omega ^{\prime }\Subset \Omega $. Then there exist a constant $%
C_{H}>0$ such that if $B_{2r}\left( x_{0}\right) \subset \Omega ^{\prime }$,
then%
\begin{equation*}
\limfunc{ess}\sup_{B_{r}}u\leq C_{H}\left( \frac{1}{\left\vert B_{2r}\left(
x_{0}\right) \right\vert }\int_{B_{2r}\left( x_{0}\right) }\left\vert
u\right\vert ^{2}dx\right) ^{\frac{1}{2}}.
\end{equation*}%
where the balls $B_{r}$ are the subunit metric balls of the metric induced
by the system of vector fields $\left\{ a_{j}\left( x\right) \frac{\partial 
}{\partial x_{j}}\right\} _{j=1}^{n}$. Moreover, if $u$ is nonnegative, then
we also have that%
\begin{equation*}
\limfunc{ess}\sup_{B_{r}}u\leq C_{H}~\limfunc{ess}\inf_{B_{r}}u.
\end{equation*}
\end{theorem}

In view of Example \ref{example-diag}, is easy to check that the operator in
Theorem \ref{theorem-noniso} satisfies the hypotheses of Theorem \ref%
{theorem-noniso-2}.

\begin{proof}[Proof of Theorem \protect\ref{theorem-noniso-2}]
For $u\in \mathcal{D}\left( L^{\frac{1}{2}}\right) $ in $\Omega \subset 
\mathbb{R}^{n}$, let $U\left( x,y\right) $ be the extension of $u$ given by
Theorem \ref{theorem-ABGR}, i.e.%
\begin{equation*}
U\left( x,y\right) =\frac{1}{\Gamma \left( \sigma \right) }\int_{0}^{\infty
}e^{-t\mathcal{L}}\left( \left( tT\right) ^{\sigma }u\left( x\right) \right)
e^{-\frac{y^{2}}{4t}}\frac{dt}{t},\qquad y>0.
\end{equation*}

From (\ref{equation}) it readily follows that $U$ satisfies the equation 
\begin{equation*}
-LU\left( x,y\right) +\frac{\partial }{\partial y^{2}}U\left( x,y\right)
=0\qquad \text{in }\Omega \times \left( 0,\infty \right) ,
\end{equation*}

which can be written in the form%
\begin{eqnarray*}
\mathcal{L}U\left( x,y\right) &=&-\func{div}\mathcal{B}\nabla U\left(
x,y\right) +\sum_{i=1}^{N}b_{i}\left( x\right) R_{i}\left( x\right) U\left(
x,y\right) \\
&&+\sum_{i=1}^{N}S_{i}^{\prime }\left( x\right) c_{i}\left( x\right) U\left(
x,y\right) +d\left( x\right) U\left( x,y\right) =0,
\end{eqnarray*}%
where 
\begin{equation*}
\mathcal{B}\left( x,y\right) =\mathcal{B}\left( x\right) =\left( 
\begin{array}{cc}
\mathbf{B}\left( x\right) & 0 \\ 
0 & 1%
\end{array}%
\right)
\end{equation*}%
satisfies Condition \ref{condition-D}, since $\mathbf{B}$ does. Now, by
assumption we have that $\mathcal{L}^{\frac{1}{2}}u=0$ in $\Omega ^{\prime
}\Subset \Omega $. We extend $U$ from $\Omega ^{\prime }\times \left(
0,\infty \right) $ to $\Omega ^{\prime }\times \mathbb{R}$ as an even
function as before%
\begin{equation*}
\widetilde{U}\left( x,y\right) =\left\{ 
\begin{array}{lr}
U\left( x,y\right) & \qquad x\in \Omega ^{\prime },~y\geq 0 \\ 
U\left( x,-y\right) & \qquad x\in \Omega ^{\prime },~y<0%
\end{array}%
\right. .
\end{equation*}%
The proof that the extended function is a solution of $\mathcal{L}\tilde{U}%
=0 $ in $\Omega ^{\prime }\times \mathbb{R}$ is similar to the proof for
weighted elliptic operators provided in Section \ref%
{section-weighted-elliptic}; we point out that the crucial part of this
proof is dealing with the principal term of the operator, which structurally
is included in the operators considered in Section \ref%
{section-weighted-elliptic}. We omit the details.

An application of Theorem \ref{theorem-SW} to this solution and the fact
that $u\left( x\right) =U\left( x,0\right) $ finishes the proof of Theorem %
\ref{theorem-noniso-2}.
\end{proof}

\subsection{Nondivergence Form Operators\label{section-nond}}

For the type of operators (\ref{L-nond}) some extra hypotheses are required
to guarantee existence and uniqueness of solutions to Dirichlet problems. In
the classic text \cite{gilbarg-trudinger-98} it is shown that if the
coefficients are uniformly continuous the Dirichlet $\mathfrak{L}_{\mathbf{A}%
}u=f$ problem has a unique solution in any $C^{1,1}$ bounded domain if $f\in
L^{p}$ for $p>n/2$ and the boundary values are continuous. The most general
existence and a-priori regularity results known today requires that the
coefficients $a^{ij}$ have small BMO norm \cite{ch-fr-lo-93,rios-03}.

In \cite{duong-yan-02} Duong and Yan prove the following result about the
resolvent set of $\mathfrak{L}_{\mathbf{A}}$ (Lemma 3.1 in \cite%
{duong-yan-02}):

\begin{lemma}
\label{lemma-sect-nond}Let $\mathfrak{L}_{\mathbf{A}}$ be given by (\ref%
{L-nond}) satisfy (\ref{ellipticity}), and for $\omega $ given by (\ref%
{omega}) let $\theta \in \left( \omega ,\pi \right] $. Then there exist
positive constants $\varepsilon _{0}$ and $C_{\theta }$ such that if $%
\sup_{1\leq i,j\leq n}\left\Vert a^{ij}\right\Vert _{\ast }<\varepsilon _{0}$
then $-\Sigma _{\pi -\theta }\subset \rho \left( \mathfrak{L}_{\mathbf{A}%
}\right) $ and%
\begin{equation*}
\left\vert z\right\vert \left\Vert \left( z-\mathfrak{L}_{\mathbf{A}}\right)
^{-1}\right\Vert _{\mathcal{B}\left( L^{p}\left( \mathbb{R}^{n}\right)
\right) }\leq C_{\theta }\qquad \text{for all }z\in -\Sigma _{\pi -\theta }.
\end{equation*}%
Moreover $\mathfrak{L}_{\mathbf{A}}$ is one-to-one and it has dense range.
\end{lemma}

In fact, if the BMO norm of the coefficients is small enough it follows that 
$\mathfrak{L}_{\mathbf{A}}$ is indeed an operator of type $\omega $ as
described in Definition \ref{def-type-w}. It only remains to show that $%
\mathfrak{L}_{\mathbf{A}}$ is closed. This follows from the a-priori
estimates in \cite{ch-fr-lo-91} obtained for VMO coefficients, and the fact
that only small BMO norm was used in their proofs (see \cite%
{ch-fr-lo-93,rios-03}). Indeed, from these papers it follows that the
a-priori estimates obtained for operators with uniformly continuous
coefficients in \cite{gilbarg-trudinger-98} (Theorem 9.11) hold for
operators with small BMO norm. More precisely:

\begin{theorem}[\protect\cite{ch-fr-lo-93}]
\label{theorem-nd-local}or $\mathfrak{L}_{\mathbf{A}}$ be given by (\ref%
{L-nond}) satisfy (\ref{ellipticity}), if $\Omega $ is any open bounded set
in $\mathbb{R}^{n}$ and $1\leq p<\infty $, there exists $\varepsilon
_{0}=\varepsilon _{0}\left( n,\lambda ,\Lambda ,p,\Omega \right) >0$ such
that if $\max_{1\leq i,j\leq n}\left\Vert a^{ij}\right\Vert _{\ast }\leq
\varepsilon _{0}$ in $\Omega $, and $u\in W_{\limfunc{loc}}^{2,p}\left(
\Omega \right) \bigcap L^{p}\left( \Omega \right) $, is such that $\mathfrak{%
L}_{\mathbf{A}}u=f\in L^{p}\left( \Omega \right) $, then for any $\Omega
^{\prime }\Subset \Omega $%
\begin{equation*}
\left\Vert u\right\Vert _{W^{2,p}\left( \Omega ^{\prime }\right) }\leq
C\left( \left\Vert u\right\Vert _{L^{p}\left( \Omega \right) }+\left\Vert
f\right\Vert _{L^{p}\left( \Omega \right) }\right) ,
\end{equation*}%
where $C$ depends on $n,\lambda ,\Lambda ,\varepsilon _{0}$,$p$, $\func{diam}%
\Omega ^{\prime }$ and $\limfunc{dist}\left( \Omega ^{\prime },\partial
\Omega \right) $.
\end{theorem}

The local estimates can be extended globally to all of $\mathbb{R}^{n}$ if $%
u $ and $f\in L^{p}\left( \mathbb{R}^{n}\right) $ and the BMO-norm of the
coefficients is small enough in cubes. Indeed, it suffices to cover $\mathbb{%
R}^{n}$ with a grid of closed unit cubes $\left\{ Q_{i}\right\}
_{i=1}^{\infty }$ and let $\widetilde{Q_{i}}$ denote the union of $Q_{i}$
with its $3^{n}-1$ immediate adjacent cubes. Then for each $Q_{i}$ we have
that%
\begin{eqnarray*}
\left\Vert u\right\Vert _{W^{2,p}\left( Q_{i}\right) }^{p} &=&\left\Vert
u\right\Vert _{L^{p}\left( Q_{i}\right) }^{p}+\left\Vert Du\right\Vert
_{L^{p}\left( Q_{i}\right) }^{p}+\left\Vert D^{2}u\right\Vert _{L^{p}\left(
Q_{i}\right) }^{p} \\
&\leq &C^{p}\left( \left\Vert u\right\Vert _{L^{p}\left( \widetilde{Q_{i}}%
\right) }^{p}+\left\Vert f\right\Vert _{L^{p}\left( \widetilde{Q_{i}}\right)
}^{p}\right)
\end{eqnarray*}%
where $C=C\left( n,\lambda ,\Lambda ,\varepsilon _{0},p\right) $ as in
Theorem \ref{theorem-nd-local}, is independent of each cube. Summing in $i$
and using that the dilated cubes $\widetilde{Q_{i}}$ have finite overlapping
yields%
\begin{equation}
\left\Vert u\right\Vert _{W^{2,p}\left( \mathbb{R}^{n}\right) }\leq C\left(
\left\Vert u\right\Vert _{L^{p}\left( \mathbb{R}^{n}\right) }+\left\Vert 
\mathfrak{L}_{\mathbf{A}}u\right\Vert _{L^{p}\left( \mathbb{R}^{n}\right)
}\right) .  \label{global-nd-est}
\end{equation}%
From this global estimate it follows that $\mathfrak{L}_{\mathbf{A}}$ is
closed, and therefore $\mathfrak{L}_{\mathbf{A}}:W^{2,p}\left( \mathbb{R}%
^{n}\right) \rightarrow L^{p}\left( \mathbb{R}^{n}\right) $ is surjective by
Lemma \ref{lemma-sect-nond}. This proves that under the extra hypothesis on
the coefficients of having small enough BMO norm, the operator is of type $%
\omega $, and hence it has a functional calculus and, in particular, the
fractional powers $\mathfrak{L}_{\mathbf{A}}^{\sigma }$ are well defined for 
$0<\sigma <1$.

We note that the global estimates (\ref{global-nd-est}) are false in general
without the small BMO norm assumption. In \cite{dongkim-14} the authors
showed that for each $1<p<\infty $ there exist an operator in $\mathbb{R}%
^{2} $ with constant coefficients in each quadrant such that (\ref%
{global-nd-est}) does no hold.

Applying the extension Theorem \ref{theorem-ABGR} to any $u\in \mathcal{D}%
\left( \mathfrak{L}_{\mathbf{A}}^{\sigma }\right) $, $0<\sigma <1$, we have
that the function 
\begin{equation}
U_{\sigma }\left( x,y\right) =\frac{1}{\Gamma \left( \sigma \right) }%
\int_{0}^{\infty }e^{-t\mathfrak{L}_{\mathbf{A}}}\left( \left( t\mathfrak{L}%
_{\mathbf{A}}\right) ^{\sigma }u\left( x\right) \right) e^{-\frac{y^{2}}{4t}}%
\frac{dt}{t}  \label{eq-Us}
\end{equation}%
satisfies $U_{\sigma }\left( \cdot ,y\right) \in \mathcal{D}\left( \mathfrak{%
L}_{\mathbf{A}}\right) =W^{2,p}\left( \mathbb{R}^{n}\right) $ for all $y>0$
and it is a solution of the initial value problem: $U_{\sigma }\left(
x,0\right) =u\left( x\right) $, and%
\begin{equation}
-\mathfrak{L}_{\mathbf{A}}U_{\sigma }\left( x,y\right) +\frac{1-2\sigma }{y}%
\frac{\partial }{\partial y}U_{\sigma }\left( x,y\right) +\frac{\partial ^{2}%
}{\partial y^{2}}U_{\sigma }\left( x,y\right) =0\qquad \text{in }\mathbb{R}%
^{n}\times \left( 0,\infty \right) ;  \label{ext-nd}
\end{equation}%
with the bounds%
\begin{equation}
\left\Vert \frac{\partial ^{m}}{\partial y^{m}}U_{\sigma }\left( \cdot
,y\right) \right\Vert _{L^{p}\left( \mathbb{R}^{n}\right) }\leq \frac{C_{m}}{%
y^{m}}\left\Vert u\right\Vert _{L^{p}\left( \mathbb{R}^{n}\right) },\qquad
m=0,1,\dots ;y>0.  \label{est-yy}
\end{equation}

From Theorem \ref{theorem-ABGR} we also have the estimates%
\begin{eqnarray}
\frac{\Gamma \left( -\sigma \right) }{4^{\sigma }\Gamma \left( \sigma
\right) }\mathfrak{L}_{\mathbf{A}}^{\sigma }u &=&\frac{1}{2\sigma }%
\lim_{y\rightarrow 0^{+}}y^{1-2\sigma }\frac{\partial }{\partial y}U_{\sigma
}\left( \cdot ,y\right)  \notag \\
&=&\lim_{y\rightarrow 0^{+}}\frac{U_{\sigma }\left( \cdot ,y\right) -u\left(
\cdot \right) }{y^{2\sigma }};  \label{first-der}
\end{eqnarray}%
and%
\begin{equation}
\frac{2\sigma -1}{4^{\sigma }}\frac{\Gamma \left( -\sigma \right) }{\Gamma
\left( \sigma \right) }\mathfrak{L}_{\mathbf{A}}^{\sigma }u=\frac{1}{2\sigma 
}\lim_{y\rightarrow 0^{+}}y^{2-2\sigma }\frac{\partial ^{2}}{\partial y^{2}}%
U\left( \cdot ,y\right) .  \label{second-der}
\end{equation}%
where the convergence is in $L^{p}\left( \mathbb{R}^{n}\right) $. Note that (%
\ref{ext-nd}) holds in the strong sense in $\mathbb{R}^{n}\times \left(
0,\infty \right) $ because of the global estimates (\ref{global-nd-est}) and
(\ref{est-yy}). Indeed, by (\ref{est-yy}) and (\ref{ext-nd}) it follows that%
\begin{equation*}
U_{\sigma }\left( \cdot ,y\right) ,~\mathfrak{L}_{\mathbf{A}}U_{\sigma
}\left( \cdot ,y\right) \in L^{p}\left( \mathbb{R}^{n}\right) \qquad \text{%
for all }y>0,
\end{equation*}%
hence (\ref{global-nd-est}) gives that $U_{\sigma }\left( \cdot ,y\right)
\in W^{2,p}\left( \mathbb{R}^{n}\right) $ for all $y>0$. Then we also have
that $\Delta U_{\sigma }\in L_{\limfunc{loc}}^{p}\left( \mathbb{R}^{n}\times
\left( 0,\infty \right) \right) $ and therefore $U_{\sigma }\in W_{\limfunc{%
loc}}^{2,p}\left( \mathbb{R}^{n}\times \left( 0,\infty \right) \right) $.
Then from (\ref{ext-nd}), The Sobolev embeddings, the local estimates in
Theorem \ref{theorem-nd-local}, and a bootstrapping argument it follows that%
\begin{equation}
U_{\sigma }\in W_{\limfunc{loc}}^{2,q}\left( \mathbb{R}^{n}\times \left(
0,\infty \right) \right) \qquad \text{for all }1<q<\infty .  \label{eq-Us-q}
\end{equation}%
And Morrey's inequality implies that $U_{\sigma }\in C^{1,\alpha }\left( 
\mathbb{R}^{n}\times \left( 0,\infty \right) \right) $ for all $0<\alpha <1$.

\begin{proposition}
\label{theorem-nd-extension}Let $\mathfrak{L}_{\mathbf{A}}$ be given by (\ref%
{L-nond}) satisfy (\ref{ellipticity}) with 
\begin{equation*}
\sup_{1\leq i,j\leq n}\left\Vert a^{ij}\right\Vert _{\ast }<\varepsilon _{0}
\end{equation*}%
for $\varepsilon _{0}$ as in Lemma \ref{lemma-sect-nond}, so that for $%
\omega $ given by (\ref{omega}) $\mathfrak{L}_{\mathbf{A}}$ is of type $%
\omega $. Let $\Omega \subset \mathbb{R}^{n}$ be a nonempty open set,
suppose $u\in \mathcal{D}\left( \mathfrak{L}_{\mathbf{A}}^{\sigma }\right)
\subset L^{p}\left( \mathbb{R}^{n}\right) $, for some $0<\sigma <1$ and $%
1<p<\infty $, and suppose that $u$ satisfies $\mathfrak{L}_{\mathbf{A}%
}^{\sigma }u=0$ in $\Omega $. Then there exists a function $V$ in $\Omega
\times \mathbb{R}$ satisfying 
\begin{equation*}
V,V_{z}\in L_{\limfunc{loc}}^{p}\left( \Omega \times \left( \mathbb{R}%
\right) \right) \quad \text{and}\qquad V\in W_{\limfunc{loc}}^{2,q}\left(
\Omega \times \left( \mathbb{R}\backslash \left\{ 0\right\} \right) \right)
\end{equation*}%
for all $1<q<\infty $, such that $V\left( \cdot ,0\right) =u$ in $%
L^{p}\left( \Omega \right) $ and $V$ is a strong solution of the problem%
\begin{equation}
-\mathfrak{L}_{\mathbf{A}}V\left( x,z\right) +z^{2-\frac{1}{\sigma }}\frac{%
\partial ^{2}}{\partial z^{2}}V\left( x,z\right) =0  \label{ext-nd-2}
\end{equation}%
in$~\Omega \times \mathbb{R}\backslash \left\{ 0\right\} $ such that $%
V\left( \cdot ,z\right) \rightarrow u$ as $z\rightarrow 0$ in $L^{p}\left(
\Omega \right) $.
\end{proposition}

\begin{proof}
Let $U$ be given by (\ref{eq-Us}), and set $\tilde{V}\left( x,z\right)
=U\left( x,2\sigma z^{\frac{1}{2\sigma }}\right) =U\left( x,y\right) $,
performing the change of variables $z=\left( \frac{y}{2\sigma }\right)
^{2\sigma }$ as in \cite{caffarelli-silvestre-07}. Equation (\ref{ext-nd})
becomes (\ref{ext-nd-2}), which holds in the strong sense in $\mathbb{R}%
^{n}\times \left( 0,\infty \right) $ since, by (\ref{eq-Us-q}), we have that 
$\tilde{V}\in W_{\limfunc{loc}}^{2,q}\left( \mathbb{R}^{n}\times \left(
0,\infty \right) \right) $ for all $1<q<\infty $.

By (\ref{first-der}) we have that%
\begin{equation*}
\tilde{V}_{z}\left( \cdot ,z\right) =\left( 2\sigma \right) ^{2\sigma
-1}y^{1-2\sigma }U_{y}\left( \cdot ,y\right) \rightarrow \frac{\left(
2\sigma \right) ^{2\sigma }\Gamma \left( -\sigma \right) }{4^{\sigma }\Gamma
\left( \sigma \right) }\mathfrak{L}_{\mathbf{A}}^{\sigma }u
\end{equation*}%
in $L^{p}\left( \mathbb{R}^{n}\right) $ as $z\rightarrow 0^{+}$. Now we set%
\begin{equation*}
V\left( x,z\right) =\left\{ 
\begin{array}{cc}
V\left( x,z\right) & \qquad x\in \Omega ,~z\geq 0 \\ 
V\left( x,-z\right) & \qquad x\in \Omega ,~z<0%
\end{array}%
\right. ,
\end{equation*}%
and since from $\mathfrak{L}_{\mathbf{A}}^{\sigma }u\equiv 0$ in $\Omega $,
it follows that $V_{z}\left( \cdot ,z\right) \rightarrow 0$ in $L^{p}\left(
\Omega \right) $ as $z\rightarrow 0$. Hence $\frac{\partial V}{\partial z}$
extends to $\Omega \times \mathbb{R}$ as an $L^{p}$ function on any bounded
strip $\Omega \times \left( -N,N\right) $. By (\ref{est-yy}) and Theorem \ref%
{theorem-ABGR} we have that $V\in L_{\limfunc{loc}}^{p}\left( \Omega \times
\left( \mathbb{R}\right) \right) $ and that $V\left( \cdot ,z\right)
\rightarrow u$ as $z\rightarrow 0$ in $L^{p}\left( \Omega \right) $.
\end{proof}

Theorem \ref{theorem-nond} is now a consequence of this result.

\begin{proof}[Proof or Theorem \protect\ref{theorem-nond}]
Let $U=V_{\sigma }\in W_{\limfunc{loc}}^{2,q}\left( \Omega \times \left( 
\mathbb{R}\backslash \left\{ 0\right\} \right) \right) $ be as in
Proposition \ref{theorem-nd-extension}. By Theorem \ref{theorem-ABGR}, (\ref%
{ext-nd-2}), and the hypothesis $u\in \mathcal{D}\left( \mathfrak{L}_{%
\mathbf{A}}\right) $, it follows that%
\begin{equation*}
\lim_{z\rightarrow 0}z^{2-\frac{1}{\sigma }}U_{zz}\left( \cdot ,z\right)
=\lim_{y\rightarrow 0}\mathfrak{L}_{\mathbf{A}}U\left( \cdot ,y\right) =%
\mathfrak{L}_{\mathbf{A}}u\in L^{p}\left( \Omega \right)
\end{equation*}%
where the limit is in $L^{p}\left( \mathbb{R}^{n}\right) $. From $0<\sigma <%
\frac{p}{p+1}$ it follows that 
\begin{equation*}
\left( 2-\frac{1}{\sigma }\right) \frac{p}{p-1}<1,
\end{equation*}%
Let $1<r<p$ such that $r\left( 2-\frac{1}{\sigma }\right) \frac{p}{p-r}<1$.
For each $N>0$ if $\Omega ^{\prime }\Subset \Omega $ we have%
\begin{eqnarray*}
&&\left( \int\limits_{\Omega ^{\prime }}\int\limits_{-N}^{N}\left\vert
U_{zz}\left( x,z\right) \right\vert ^{r}~dz~dx\right) ^{\frac{1}{r}} \\
&\leq &\left( \int\limits_{\Omega ^{\prime }}\int\limits_{-N}^{N}\left\vert
z^{2-\frac{1}{\sigma }}U_{zz}\left( x,z\right) \right\vert ^{p}~dz~dx\right)
^{\frac{1}{p}}\left( \int\limits_{\Omega ^{\prime
}}\int\limits_{-N}^{N}z^{-\left( 2-\frac{1}{\sigma }\right) r\frac{p}{p-r}%
}~dz~dx\right) ^{\frac{p-r}{rp}} \\
&\leq &C_{N,\sigma ,p,\Omega ^{\prime }}\left( \int\limits_{\Omega ^{\prime
}}\int\limits_{0}^{N}\left\vert z^{2-\frac{1}{\sigma }}U_{zz}\left(
x,z\right) \right\vert ^{p}~dz~dx\right) ^{\frac{1}{p}}<\infty .
\end{eqnarray*}%
Thus, $U_{zz}$ extends as an $L_{\limfunc{loc}}^{r}$ function in all bounded
strips $\Omega \times \left( -N,N\right) $, $N>0$. By proposition \ref%
{theorem-nd-extension} we already have that $U_{z}\in $ $L_{\limfunc{loc}%
}^{r}\left( \Omega \times \left( -N,N\right) \right) $. Moreover, by the
local estimates in Theorem \ref{theorem-nd-local} we have that $U\left(
\cdot ,z\right) \in W_{\limfunc{loc}}^{2,p}\left( \Omega \right) \subset W_{%
\limfunc{loc}}^{2,r}\left( \Omega \right) $ for all $z$, with locally
uniform bounds for bounded $z$. Hence $\Delta U\in L_{\limfunc{loc}%
}^{r}\left( \Omega \times \mathbb{R}\right) $ and consequently $U\in W_{%
\limfunc{loc}}^{2,r}\left( \Omega \times \mathbb{R}\right) $. Then by
Theorem \ref{theorem-nd-local}, the Sobolev's embedding, and a bootstrapping
argument we conclude that 
\begin{equation*}
U\in W_{\limfunc{loc}}^{2,q}\left( \Omega \times \mathbb{R}\right) \qquad 
\text{for all }1<q<\infty .
\end{equation*}%
Then Morrey's inequality implies that $U\in C^{1,\alpha }\left( \Omega
\times \mathbb{R}\right) $ for all $0<\alpha <1$. Since $u\left( z\right)
=U\left( z,0\right) $ we conclude that $u\in C^{1,\alpha }\left( \Omega
\right) $.
\end{proof}

\section{Appendix\label{section-appendix}}

Let $\mathcal{X}$ be a Banach space; $\mathfrak{L}\left( \mathcal{X}\right) $
denotes the algebra of bounded linear operators on $\mathcal{X}$. Given a
linear operator $T$ on $\mathcal{X}$, the resolvent set $\rho \left(
T\right) $ is the set of $\lambda \in \mathbb{C}$ such that $T-\lambda $ is
one to one and $\mathsf{R}_{T}\left( \lambda \right) =\left( T-\lambda
\right) ^{-1}\in \mathfrak{L}\left( \mathcal{X}\right) $; $\mathsf{R}%
_{T}\left( \lambda \right) $ is called the resolvent the operator of $T$ at $%
\lambda $. The spectrum of $T$, $\sigma \left( T\right) $ is the complement
of $\rho \left( T\right) $ in $\mathbb{C}$, together with $\infty $ if $T$
is not bounded.

We consider closed operators $T:\mathcal{D}\left( T\right) \subset \mathcal{%
X\rightarrow X}$ where $\mathcal{X}$ is a Banach space. Such operators have
a holomorphic functional calculus. We denote by $\mathcal{C\ell }\left( 
\mathcal{X}\right) $ the set of all closed operators on $\mathcal{X}$; note
that $\mathcal{L}\left( \mathcal{X}\right) \subset \mathcal{C\ell }\left( 
\mathcal{X}\right) $.

\subsection{Non-symmetric Dirichlet forms}

Dirichlet forms can be defined in general Hilbert spaces, but for our
applications it suffices to consider $L^{2}$ spaces. Specifically, let $X$
be a locally compact metric space and $\mu $ is a $\sigma $-finite positive
Radon measure on $X$ such that $\limfunc{support}\mu =X$. We will work on
the real Hilbert space $L^{2}\left( X,\mu \right) $ with the usual $L^{2}$%
-inner product $\left\langle \cdot ,\cdot \right\rangle $, and in this
context $\left\Vert f\right\Vert $ denotes the $L^{2}$-norm $\left\langle
f,f\right\rangle ^{1/2}$. The basics of non-symmetric Dirichlet forms
presented here can be found in chapter 1 of \cite{oshima13}; for symmetric
Dirichlet forms see \cite{fukushima-oshima-takeda-11}.

A bilinear form $\mathcal{E}$ with domain $\mathcal{D}\left[ \mathcal{E}%
\right] =\mathcal{F}\subset L^{2}\left( X,\mu \right) $ is a function $%
\mathcal{E}:\mathcal{F}\times \mathcal{F}\rightarrow \mathbb{R}$ which is
linear in each variable separately.

\begin{definition}
\label{def-Dirichlet}A bilinear form $\mathcal{E}$ on $\mathcal{F}\subset
L^{2}\left( X,\mu \right) $ is a \emph{(semi-)Dirichlet form} on $%
L^{2}\left( X,\mu \right) $ if $\mathcal{F}$ is a dense subspace of $%
L^{2}\left( X,\mu \right) $ and the following conditions are satisfied:%
\renewcommand{\theenumi}{\Roman{enumi}}%

\begin{enumerate}
\item \label{E1-lb}$\mathcal{E}$ is \emph{lower bounded}: There exists a
nonnegative constant $\alpha _{0}$ such that 
\begin{equation*}
\mathcal{E}_{\alpha _{0}}\left( u,u\right) \geq 0\qquad \text{for all }u\in 
\mathcal{F},
\end{equation*}%
where $\mathcal{E}_{\alpha _{0}}\left( u,v\right) =\mathcal{E}\left(
u,v\right) +\alpha _{0}\left\langle u,v\right\rangle $.

\item \label{E2-sector}$\mathcal{E}$ \emph{satisfies the sector condition: }%
There exists a constant $K\geq 1$ such that%
\begin{equation*}
\left\vert \mathcal{E}\left( u,v\right) \right\vert \leq K\mathcal{E}%
_{\alpha _{0}}\left( u,u\right) ^{1/2}\mathcal{E}_{\alpha _{0}}\left(
v,v\right) ^{1/2}\qquad \text{for all }u,v\in \mathcal{F}\text{.}
\end{equation*}

\item \label{E3-complete}$\mathcal{F}$ is a Hilbert space relative to the
inner product%
\begin{equation*}
\mathcal{E}_{\alpha }^{\left( s\right) }\left( u,v\right) =\frac{1}{2}\left( 
\mathcal{E}_{\alpha }\left( u,v\right) +\mathcal{E}_{\alpha }\left(
v,u\right) \right) \qquad \text{for all }\alpha >\alpha _{0}.
\end{equation*}

\item \label{E4-Markov}$\mathcal{E}$ \emph{satisfies the Markov property:}
for all $u\in \mathcal{F}$ and $a\geq 0$, then $u\wedge a\in \mathcal{F}$ ,%
\begin{equation*}
\mathcal{E}\left( u\wedge a,u-u\wedge a\right) \geq 0.
\end{equation*}
\end{enumerate}
\end{definition}

\bigskip Note that for $\alpha >\alpha _{0}$ we have, with $K$ as in (\ref%
{E3-complete}) and $K_{\alpha }=K+\frac{\alpha }{\alpha -\alpha _{0}}$,%
\begin{equation*}
\left\vert \mathcal{E}_{\alpha }\left( u,v\right) \right\vert \leq K_{\alpha
}\mathcal{E}_{\alpha }\left( u,u\right) ^{1/2}\mathcal{E}_{\alpha }\left(
v,v\right) ^{1/2}\qquad \text{for all }u,v\in \mathcal{F}\text{.}
\end{equation*}%
In particular, $\mathcal{E}_{\alpha }$ and $\mathcal{E}_{\beta }$ determine
equivalent metrics for any fixed $\alpha ,\beta >0$.

When $\alpha _{0}=0$ in the above definition we say that $\mathcal{E}$ is a 
\emph{nonnegative Dirichlet form}. If a nonnegative Dirichlet form $\mathcal{%
E}$ also satisfies 
\begin{equation}
\left( u-u\wedge a,u\wedge a\right) \geq 0\qquad \text{for all }u\in 
\mathcal{F}\text{ and }a\geq 0,  \label{E5-Markov-conj}
\end{equation}%
then we say that $\mathcal{E}$ is a \emph{non-symmetric Dirichlet form}. If
a non-symmetric Dirichlet form satisfies $\mathcal{E}\left( u,v\right) =%
\mathcal{E}\left( v,u\right) $ for all $u,v\in \mathcal{=F}$ then $\mathcal{E%
}$ is called \emph{symmetric Dirichlet form}.

The framework of Dirichlet forms includes the first two applications that we
will present in this work. Associated to each Dirichlet form $\mathcal{E}$
there is an operator $-L_{\mathcal{E}}$ which is the generator of a strongly
continuous semigroup $e^{-tL_{\mathcal{E}}}$. In fact this result is true
for forms that are just closed, the following theorem can be found in \cite%
{oshima13} (Theorem 1.1.2).

\begin{theorem}
\label{theorem-Dirichlet}Suppose $\mathcal{E}$ is a bilineal form with dense
domain $\mathcal{F}\subset L^{2}\left( X,\mu \right) $, and which satisfies (%
\ref{E1-lb}), (\ref{E2-sector}), and (\ref{E3-complete}) from Definition \ref%
{def-Dirichlet}. Then there exist strongly continuous semigroups $\left\{
T_{t}\right\} _{t>0}$ and $\left\{ \widehat{T}_{t}\right\} _{t>0}$ on $%
L^{2}\left( X,\mu \right) $ such that $\left\Vert T_{t}\right\Vert \leq
e^{\alpha _{0}t}$, $\left\Vert \widehat{T}_{t}\right\Vert \leq e^{\alpha
_{0}t}$, $\left\langle T_{t}f,g\right\rangle =\left\langle f,\widehat{T}%
_{t}g\right\rangle $ and whose resolvents%
\begin{equation*}
G_{\alpha }=\int_{0}^{\infty }e^{-\alpha t}T_{t}~dt\qquad \text{and}\qquad 
\widehat{G}_{\alpha }=\int_{0}^{\infty }e^{-\alpha t}\widehat{T}_{t}~dt
\end{equation*}%
satisfy%
\begin{equation*}
\mathcal{E}_{\alpha }\left( G_{\alpha }f,u\right) =\left\langle
f,u\right\rangle =\mathcal{E}_{\alpha }\left( u,\widehat{G}_{\alpha
}f\right) ,
\end{equation*}%
for all $f\in L^{2}\left( X,\mu \right) $, $u\in \mathcal{F}$, and $\alpha
>0 $. Moreover, $T_{t}=e^{-tL_{\mathcal{E}}}$ and $\widehat{T}_{t}=e^{-t%
\widehat{L}_{\mathcal{E}}}$ where the generators $L_{\mathcal{E}}$ and $%
\widehat{L}_{\mathcal{E}},$ also called the associated operator to $\mathcal{%
E}$ and the associated adjoint operator to $\mathcal{E}$, respectively, have
domains $\mathcal{D}\left( L_{\mathcal{E}}\right) \subset \mathcal{F}$ and $%
\mathcal{D}\left( \widehat{L}_{\mathcal{E}}\right) \subset \mathcal{F}$
which are dense in $L^{2}\left( X,\mu \right) $; for all $\alpha >\alpha
_{0} $ and $f\in L^{2}\left( X,\mu \right) $ we have 
\begin{equation*}
G_{\alpha }f=\left( \alpha -\alpha _{0}+L_{\mathcal{E}}\right) ^{-1}f\qquad 
\text{and}\qquad \widehat{G}_{\alpha }f=\left( \alpha -\alpha _{0}+\widehat{L%
}_{\mathcal{E}}\right) ^{-1}f;
\end{equation*}%
with the bounds%
\begin{equation}
\left\Vert G_{\alpha }\right\Vert \leq \frac{1}{\alpha -\alpha _{0}}\qquad 
\text{and}\qquad \left\Vert \widehat{G}_{\alpha }\right\Vert \leq \frac{1}{%
\alpha -\alpha _{0}}.  \label{eq-resolvent}
\end{equation}%
Finally, for all $u\in \mathcal{D}\left( L_{\mathcal{E}}\right) $, $v\in 
\mathcal{D}\left( \widehat{L}_{\mathcal{E}}\right) $, and $f\in \mathcal{F}$
we have the identities%
\begin{equation}
\left\langle L_{\mathcal{E}}u,f\right\rangle =\mathcal{E}_{\alpha
_{0}}\left( u,f\right) ,\text{\qquad and\qquad }\left\langle f,\widehat{L}_{%
\mathcal{E}}v\right\rangle =\mathcal{E}_{\alpha _{0}}\left( f,v\right) \text{%
.}  \label{eq-operators}
\end{equation}%
In this case we also have $\left\langle L_{\mathcal{E}}u,v\right\rangle =%
\mathcal{E}_{\alpha _{0}}\left( u,v\right) =\left\langle u,\widehat{L}_{%
\mathcal{E}}v\right\rangle .$
\end{theorem}

Note that since $\mathcal{F}$ is dense in $L^{2}\left( X,\mu \right) $, the
operators $L_{\mathcal{E}}$ and $\widehat{L}_{\mathcal{E}}$ are
characterized by (\ref{eq-operators}). That is, if $h\in $ $L^{2}\left(
X,\mu \right) $ and $\left\langle h,f\right\rangle =\mathcal{E}_{\alpha
_{0}}\left( u,f\right) $ for some $u\in \mathcal{D}\left( L_{\mathcal{E}%
}\right) $ and all $f\in \mathcal{F}$, then $L_{\mathcal{E}}u=h$. Moreover,
because of (\ref{eq-operators}) and the completeness assumption (\ref%
{E3-complete}) we have that $L_{\mathcal{E}}$ is closed; i.e. if $u_{n}\in 
\mathcal{D}\left( L_{\mathcal{E}}\right) $, $u_{n}\rightarrow u$ in $%
L^{2}\left( X,\mu \right) $, and $f_{n}=L_{\mathcal{E}}u_{n}\rightarrow f$
in $L^{2}\left( X,\mu \right) $, then $u\in \mathcal{D}\left( L_{\mathcal{E}%
}\right) $ and $Lu=f$. Equivalently, $L_{\mathcal{E}}$ is closed if and only
if $\mathcal{D}\left( L_{\mathcal{E}}\right) $ is a Banach space with the
norm $\left\Vert u\right\Vert _{\mathcal{D}\left( L_{\mathcal{E}}\right)
}=\left\Vert u\right\Vert +\mathcal{E}\left( u,u\right) ^{\frac{1}{2}%
}\approx \mathcal{E}_{1}\left( u,u\right) ^{\frac{1}{2}}$, this condition is
guaranteed by (\ref{E3-complete}). Similar statements apply to $\widehat{L}_{%
\mathcal{E}}$.

\subsection{Sectorial operators and their calculus\label{section-calculus}}

All the operators we consider in our present applications are \emph{sectorial%
} operators. This type of operators was first introduced by Kato \cite%
{kato60}, but here we adopt the more general definition in which we do not
require the operator to be given by a sectorial form. Our definition is
precisely that of operators \emph{of type} $\omega $ as introduced by
McIntosh \cite{mcintosh86}, which was generalized as sectorial operators
more recently to include Banach spaces (see \cite{haase06,
bandara-mcintosh10, batty09} and references within).

Given $0\leq \omega <\pi $ we denote by $\Sigma _{\omega }$ the open complex
sector%
\begin{equation*}
\Sigma _{\omega }=\{z\in \mathbb{C}:z\neq 0,|\arg (z)|<\omega \}.
\end{equation*}

\begin{definition}
\label{def-type-w}Given $0\leq \omega <\pi $, an operator $T$ on a Banach
space $\mathcal{X}$ is said to be of \emph{type} $\omega $ , or \emph{%
sectorial} of angle $\omega $, if $T$ closed and densely defined in $%
\mathcal{X}$, $\sigma \left( T\right) \subset \overline{\Sigma _{\omega }}%
\bigcup \left\{ \infty \right\} $, and for each $\theta \in \left( \omega
,\pi \right] $ there exists a constant $c_{\theta }>0$ such that%
\begin{equation*}
\left\Vert \mathsf{R}_{T}\left( z\right) \right\Vert =\left\Vert \left(
z-T\right) ^{-1}\right\Vert _{\mathfrak{L}\left( \mathcal{X}\right) }\leq 
\frac{c_{\theta }}{\left\vert z\right\vert }
\end{equation*}%
for all non-zero $z\notin \Sigma _{\theta }$.
\end{definition}

If $T$ is a sectorial of angle $\omega $ on $\mathcal{X}$ with $0\leq \omega
<\pi /2$, the natural approach to establishing a holomorphic functional
calculus and defining $\varphi (T)$ for $\varphi \in H^{\infty }(\Sigma
_{\mu })$ is to first consider ${\varphi }$ in the smaller class $%
H_{0}^{\infty }\left( \Sigma _{\mu }\right) ,$ given by 
\begin{equation*}
H_{0}^{\infty }\left( \Sigma _{\mu }\right) =\left\{ {\varphi \in H}\left(
\Sigma _{\mu }\right) :\exists c,s>0\quad |\varphi (z)|\leq \frac{%
c\left\vert z\right\vert ^{s}}{\left( 1+\left\vert z\right\vert \right)
^{-2s}},~\forall z\in \Sigma _{\mu }\right\} .
\end{equation*}%
First, the semigroup $e^{-zT}$ existence may be established by the Cauchy
integral identity%
\begin{equation*}
e^{-zT}=\int_{\Gamma _{\alpha }}e^{-z\zeta }\mathsf{R}_{T}\left( \zeta
\right) ~d\zeta
\end{equation*}%
where $\Gamma _{\alpha }$ is the boundary of $\Sigma _{\alpha }$ with
positive orientation, and $\alpha $ is for any fixed angle such that $\omega
<\alpha <\pi /2-\arg z$. This semigroup is contractive ($\left\Vert
e^{-zT}\right\Vert \leq 1$) and holomorphic in the sector $\Sigma _{\pi
/2-\omega }$. Then we can write an integral representation of $\varphi (T)$
for any ${\varphi \in }H_{0}^{\infty }\left( \Sigma _{\mu }\right) $, with $%
\omega <\theta <\nu <\min (\mu ,\pi /2)$, namely: 
\begin{equation}
\varphi (T)=\int_{\Gamma _{\pi /2-\theta }}e^{-zT}\eta (z)\,dz,
\label{eqn:L2-holo-rep}
\end{equation}%
where 
\begin{equation}
\eta (z)=\frac{1}{2\pi i}\int_{\gamma _{\nu }(z)}e^{\zeta z}\varphi (\zeta
)\,d\zeta  \label{eqn:L2-holo-rep-eta}
\end{equation}%
with $\gamma _{\nu }(z)=\mathbb{R}^{+}e^{i\mathrm{sign}(\mathrm{Im}(z))\nu }$%
. Note that 
\begin{equation*}
|\eta (z)|\lesssim \min (1,|z|^{-s-1}),\quad z\in \Gamma _{\pi /2-\theta },
\end{equation*}%
so the representation (\ref{eqn:L2-holo-rep}) converges in $\mathcal{X}$,
and we have the bound 
\begin{equation}
\Vert \varphi (T)f\Vert \leq C\Vert \varphi \Vert _{\infty }\Vert f\Vert
,\qquad f\in H_{0}^{\infty }(\Sigma _{\mu }),  \label{eqn:L2-holo-bd}
\end{equation}%
where $\left\Vert f\right\Vert $ denotes the norm of $f$ in $\mathcal{X}$.

Now, if $T$ is an operator of type $\omega $ as above, then $T$ has an ${H}%
^{\infty }$ functional calculus and~(\ref{eqn:L2-holo-bd}) extends to all of 
$H^{\infty }\left( \Sigma _{\mu }\right) $ and also to holomorphic functions
of polynomial growth (see also~\cite{mcintosh86, couling96, haase06}). In
particular, this approach allows us to define (fractional) powers $T^{\sigma
}$ of $T$ for any $\sigma \in \mathbb{R}$. Of course, these operators will
not in general be bounded if $T$ is not bounded. The following is a
resolution of fractional powers $T^{\sigma }$, for $\sigma >0:$ 
\begin{equation}
T^{\sigma }x_{0}=\frac{1}{\Gamma \left( -\sigma \right) }\int_{0}^{\infty
}\left( e^{-tT}-1\right) x_{0}\frac{dt}{t^{1+\sigma }}.  \label{eq-fract}
\end{equation}%
This is the resolution is we adopt in Theorem \ref{theorem-ABGR}. By the
Spectral Mapping Theorem it follows that if $T$ is sectorial of angle $%
\omega $, then $T^{\sigma }$ is sectorial of angle $\sigma \omega $ for any $%
0<\sigma \leq 1$, $\overline{\mathcal{D}\left( T\right) }=\mathcal{D}\left(
T^{\sigma }\right) \supset \mathcal{D}\left( T\right) $, and $\mathcal{N}%
\left( T^{\sigma }\right) =\mathcal{N}\left( T\right) $, where $\mathcal{N}%
\left( T\right) $ denotes the kernel of $T$ (see Proposition 3.1.1. in \cite%
{haase06}).

A real bilinear form $\mathcal{E}$ with domain $\mathcal{D}\left[ \mathcal{E}%
\right] \subset \left( L^{2}\left( X,\mu \right) ,\mathbb{R}\right) $ has a
natural extension as a complex sesquilinear form $\widetilde{\mathcal{E}}$
with domain $\mathcal{D}\left[ \widetilde{\mathcal{E}}\right] =\mathcal{D}%
\left[ \mathcal{E}\right] +i\mathcal{D}\left[ \mathcal{E}\right] \subset
\left( L^{2}\left( X,\mu \right) ,\mathbb{C}\right) $ by setting 
\begin{eqnarray}
\widetilde{\mathcal{E}}\left( f_{1},f_{2}\right) &=&\widetilde{\mathcal{E}}%
\left( g_{1}+ih_{1},g_{2}+ih_{2}\right)  \label{eq-complex-sesq} \\
&=&\mathcal{E}\left( g_{1},g_{2}\right) +\mathcal{E}\left(
h_{1},h_{2}\right) +i\mathcal{E}\left( h_{1},g_{2}\right) -i\mathcal{E}%
\left( g_{1},h_{2}\right)  \notag
\end{eqnarray}%
where $g_{i}=\func{Re}f_{i}$ and $h_{i}=\func{Im}f_{i}$, $i=1,2$. Note that $%
\mathcal{E}$ is indeed the restriction of $\widetilde{\mathcal{E}}$ to $%
\mathcal{D}\left[ \mathcal{E}\right] \subset \left( L^{2}\left( X,\mu
\right) ,\mathbb{R}\right) $. Thus, if $\mathcal{E}$ is as in the previous
corollary, the sesquilinear form $\widetilde{\mathcal{E}}$ is \emph{accretive%
}, that is, $\func{Re}\widetilde{\mathcal{E}}\left( f,f\right) \geq 0$.
Moreover, if $K$ is the constant from condition (\ref{E2-sector}) in
Definition \ref{def-Dirichlet}, $\widetilde{\mathcal{E}}$ is sectorial with
the same constant:%
\begin{equation*}
\left\vert \func{Im}\widetilde{\mathcal{E}}\left( f,f\right) \right\vert
\leq K~\func{Re}\widetilde{\mathcal{E}}\left( f,f\right) .
\end{equation*}%
The operator $L_{\widetilde{\mathcal{E}}}$ associated to this sesquilinear
form (see 1.2.3 in \cite{ouhabaz05}), and its corresponding adjoint operator 
$\widehat{L}_{\widetilde{\mathcal{E}}}$ are the generators of a holomorphic
semigroup in the sector $\Sigma _{\arctan \left( 1/K\right) }$, see Theorem
1.53 in \cite{ouhabaz05} for a proof of the next result.

\begin{theorem}
\label{theorem-semigroup}Let $\mathcal{E}$ is a nonnegative bilineal form
with dense domain in $L^{2}\left( X,\mu \right) ,$ satisfying (\ref{E1-lb}),
(\ref{E2-sector}), and (\ref{E3-complete}), and let $\widetilde{\mathcal{E}}$
be the sesquilinear extension (\ref{eq-complex-sesq}). Then the associated
operators $-L_{\widetilde{\mathcal{E}}}$ and $-\widehat{L}_{\widetilde{%
\mathcal{E}}}$ generate strongly continuous semigroups $e^{-tL_{\widetilde{%
\mathcal{E}}}}$ and $e^{-t\widehat{L}_{\widetilde{\mathcal{E}}}}$, $t\geq 0$%
, on $L^{2}\left( \mathcal{X},\mu \right) $. These semigroups are
holomorphic on the sector $\Sigma _{\arctan \left( 1/K\right) }$ and the
operators $e^{-zL_{\widetilde{\mathcal{E}}}}$, $e^{-z\widehat{L}_{\widetilde{%
\mathcal{E}}}}$ are contraction operators, i.e. $\left\Vert e^{-zL_{%
\widetilde{\mathcal{E}}}}\right\Vert \leq 1$ and $\left\Vert e^{-z\widehat{L}%
_{\widetilde{\mathcal{E}}}}\right\Vert \leq 1$, for all $z\in \Sigma
_{\arctan \left( 1/K\right) }$.
\end{theorem}

\begin{proposition}
\label{prop-restriction}Let $\mathcal{E}$ is a nonnegative bilineal form
with dense domain $\mathcal{F}$ in $L^{2}\left( \mathcal{X},\mu \right) ,$
satisfying (\ref{E1-lb}), (\ref{E2-sector}), and (\ref{E3-complete}), and
let $\widetilde{\mathcal{E}}$ be the sesquilinear extension (\ref%
{eq-complex-sesq}). Then the associated operator $L_{\mathcal{E}}$ ($%
\widehat{L}_{\mathcal{E}}$) is the restriction of the operator $L_{%
\widetilde{\mathcal{E}}}$ ($\widehat{L}_{\widetilde{\mathcal{E}}}$) in the
sense that $\mathcal{D}\left( L_{\mathcal{E}}\right) =\func{Re}$ $\mathcal{D}%
\left( L_{\widetilde{\mathcal{E}}}\right) $ ($\mathcal{D}\left( \widehat{L}_{%
\mathcal{E}}\right) =\func{Re}$ $\mathcal{D}\left( \widehat{L}_{\widetilde{%
\mathcal{E}}}\right) $), and $L_{\mathcal{E}}\left( \func{Re}u\right) =\func{%
Re}\left( L_{\widetilde{\mathcal{E}}}u\right) $ ($\widehat{L}_{\mathcal{E}%
}\left( \func{Re}u\right) =\func{Re}\left( \widehat{L}_{\widetilde{\mathcal{E%
}}}u\right) $) for all $u\in \mathcal{D}\left( L_{\mathcal{E}}\right) $ ($%
\mathcal{D}\left( \widehat{L}_{\mathcal{E}}\right) $).
\end{proposition}

\begin{proof}
Suppose $u\in \mathcal{D}\left( L_{\widetilde{\mathcal{E}}}\right) $, then
for all $v\in \mathcal{D}\left[ \widetilde{\mathcal{E}}\right] =\mathcal{D}%
\left[ \mathcal{E}\right] +i\mathcal{D}\left[ \mathcal{E}\right] $ 
\begin{equation*}
\left\langle L_{\widetilde{\mathcal{E}}}u,v\right\rangle =\int_{\mathcal{X}%
}\left( L_{\widetilde{\mathcal{E}}}u\right) \overline{v}~d\mu =\widetilde{%
\mathcal{E}}\left( u,v\right) .
\end{equation*}%
Since $\mathcal{D}\left[ \widetilde{\mathcal{E}}\right] \supset \mathcal{D}%
\left[ \mathcal{E}\right] =\func{Re}\mathcal{D}\left[ \widetilde{\mathcal{E}}%
\right] $, we have that for $v\in \mathcal{D}\left[ \mathcal{E}\right] $%
\begin{equation*}
\func{Re}\left\langle L_{\widetilde{\mathcal{E}}}u,v\right\rangle =\int_{%
\mathcal{X}}\func{Re}\left( L_{\widetilde{\mathcal{E}}}u\right) v~d\mu =%
\func{Re}\widetilde{\mathcal{E}}\left( u,v\right) =\mathcal{E}\left( \func{Re%
}u,v\right) .
\end{equation*}
Thus $\func{Re}u\in \mathcal{D}\left( L_{\mathcal{E}}\right) $ and $L_{%
\mathcal{E}}\left( \func{Re}u\right) =\func{Re}\left( L_{\widetilde{\mathcal{%
E}}}u\right) $. Similarly, if $u\in \mathcal{D}\left( L_{\mathcal{E}}\right) 
$, and $v\in \mathcal{D}\left[ \widetilde{\mathcal{E}}\right] ,$ write $%
v=f+ig$ where $f=\func{Re}v$ and $g=\func{Im}v$, then%
\begin{eqnarray*}
\left\langle L_{\mathcal{E}}u,v\right\rangle &=&\left\langle L_{\mathcal{E}%
}u,f\right\rangle -i\left\langle L_{\mathcal{E}}u,g\right\rangle \\
&=&\int_{\mathcal{X}}\left( L_{\mathcal{E}}u\right) f~d\mu -i\int_{\mathcal{X%
}}\left( L_{\mathcal{E}}u\right) g~d\mu \\
&=&\mathcal{E}\left( u,f\right) -i\mathcal{E}\left( u,g\right) =\widetilde{%
\mathcal{E}}\left( u,v\right) .
\end{eqnarray*}%
So $u\in \mathcal{D}\left( L_{\widetilde{\mathcal{E}}}\right) $ and $L_{%
\mathcal{E}}u=L_{\widetilde{\mathcal{E}}}u$. The proof for the operators $%
\widehat{L}_{\mathcal{E}}$ and $\widehat{L}_{\widetilde{\mathcal{E}}}$ is
similar.
\end{proof}

As a consequence of the previous proposition and Theorem \ref%
{theorem-semigroup} we have that if $\mathcal{E}$ is a nonnegative bilineal
form as in the proposition then the operators $L_{\mathcal{E}}$ and $L_{%
\widetilde{\mathcal{E}}}$ generate strongly continuous semigroups in $\left(
-\infty ,0\right] $ and holomorphic contractive semigroups $e^{-zL_{\mathcal{%
E}}}$, $e^{-z\widehat{L}_{\mathcal{E}}}$ on the sector $\Sigma _{\arctan
\left( 1/K\right) }$. In turn, standard results imply that the operators $L_{%
\mathcal{E}}$ and $\widehat{L}_{\mathcal{E}}$ are sectorial of angle $\frac{%
\pi }{2}-\arctan \left( 1/K\right) $, see for example Theorem II.4.6 in \cite%
{engel-nagel99} for a proof. We collect these facts in the following
corollary.

\begin{corollary}
\label{coro-sectorial}Let $\mathcal{E}$ is a nonnegative bilineal form with
dense domain in $L^{2}\left( X,\mu \right) ,$ satisfying (\ref{E1-lb}), (\ref%
{E2-sector}), and (\ref{E3-complete}) from Definition \ref{def-Dirichlet}.
Then the associated operators $L_{\widetilde{\mathcal{E}}}$ and $\widehat{L}%
_{\widetilde{\mathcal{E}}}$ are sectorial of angle $\frac{\pi }{2}-\arctan
\left( 1/K\right) $, where $K$ is then constant in (\ref{E2-sector}).
Moreover, $-L_{\widetilde{\mathcal{E}}}$ and $-\widehat{L}_{\widetilde{%
\mathcal{E}}}$ generate strongly continuous semigroups $e^{-tL_{\mathcal{E}%
}} $ and $e^{-t\widehat{L}_{\mathcal{E}}}$, $t\geq 0$, on $L^{2}\left( 
\mathcal{X},\mu \right) $. These semigroups are holomorphic on the sector $%
\Sigma _{\arctan \left( 1/K\right) }$ and the operators $e^{-zL_{\widetilde{%
\mathcal{E}}}}$, $e^{-z\widehat{L}_{\widetilde{\mathcal{E}}}}$ are
contraction operators.
\end{corollary}

\bibliographystyle{amsplain}
\providecommand{\bysame}{\leavevmode\hbox to3em{\hrulefill}\thinspace}
\providecommand{\MR}{\relax\ifhmode\unskip\space\fi MR }
\providecommand{\MRhref}[2]{%
  \href{http://www.ams.org/mathscinet-getitem?mr=#1}{#2}
}
\providecommand{\href}[2]{#2}

\end{document}